\documentclass{lmcspre}
\usepackage{hyperref}

\usepackage{amsmath}
\usepackage{amssymb}
\usepackage{amsthm}

\usepackage{tikz}
\usetikzlibrary{fadings}

\def\r{\mathrm{r}}

\def\qcb{\mathsf{qcb}}

\makeatletter
\newcommand{\douwidehat}[2]{%
  \sbox0{$\m@th#1\widehat{\hphantom{#2}}$}%
  \sbox2{$\m@th#1x$}
  \sbox4{$\m@th#1#2$}
  \dimen0=\ht0
  \advance\dimen0 -.8\ht2
  \dimen2=\dp4
  \rlap{%
    \raisebox{\dimexpr\dimen0-\dimen2}{%
      \scalebox{1}[-1]{\box0}%
    }%
  }%
  {#2}%
}
\makeatother

\title{Computable Bases}

\author{Vasco Brattka\lmcsorcid{0000-0003-4664-2183}}
\address{Faculty of Computer Science, Universität der Bundeswehr München, Germany and
Department of Mathematics and Applied Mathematics, University of Cape Town, South Africa}
\email{Vasco.Brattka@cca-net.de}
\author{Emmanuel Rauzy\lmcsorcid{0009-0005-0395-3311}}
\address{Laboratoire d’Algorithmique, Complexité et Logique, Département d’Informatique, Université Paris-Est Créteil, France}
\email{emmanuel.rauzy@u-pec.fr}

\begin{document} 



\def\AA{{\mathcal A}}
\def\BB{{\mathcal B}}
\def\CC{{\mathcal C}}
\def\DD{{\mathcal D}}
\def\EE{{\mathcal E}}
\def\FF{{\mathcal F}}
\def\GG{{\mathcal G}}
\def\HH{{\mathcal H}}
\def\II{{\mathcal I}}
\def\JJ{{\mathcal J}}
\def\KK{{\mathcal K}}
\def\LL{{\mathcal L}}
\def\MM{{\mathcal M}}
\def\NN{{\mathcal N}}
\def\OO{{\mathcal O}}
\def\PP{{\mathcal P}}
\def\QQ{{\mathcal Q}}
\def\RR{{\mathcal R}}
\def\SS{{\mathcal S}}
\def\TT{{\mathcal T}}
\def\UU{{\mathcal U}}
\def\VV{{\mathcal V}}
\def\WW{{\mathcal W}}
\def\XX{{\mathcal X}}
\def\YY{{\mathcal Y}}
\def\ZZ{{\mathcal Z}}


\def\bA{{\mathbf A}}
\def\bB{{\mathbf B}}
\def\bC{{\mathbf C}}
\def\bD{{\mathbf D}}
\def\bE{{\mathbf E}}
\def\bF{{\mathbf F}}
\def\bG{{\mathbf G}}
\def\bH{{\mathbf H}}
\def\bI{{\mathbf I}}
\def\bJ{{\mathbf J}}
\def\bK{{\mathbf K}}
\def\bL{{\mathbf L}}
\def\bM{{\mathbf M}}
\def\bN{{\mathbf N}}
\def\bO{{\mathbf O}}
\def\bP{{\mathbf P}}
\def\bQ{{\mathbf Q}}
\def\bR{{\mathbf R}}
\def\bS{{\mathbf S}}
\def\bT{{\mathbf T}}
\def\bU{{\mathbf U}}
\def\bV{{\mathbf V}}
\def\bW{{\mathbf W}}
\def\bX{{\mathbf X}}
\def\bY{{\mathbf Y}}
\def\bZ{{\mathbf Z}}


\def\IB{{\mathbb{B}}}
\def\IC{{\mathbb{C}}}
\def\IF{{\mathbb{F}}}
\def\IN{{\mathbb{N}}}
\def\IP{{\mathbb{P}}}
\def\IQ{{\mathbb{Q}}}
\def\IR{{\mathbb{R}}}
\def\IS{{\mathbb{S}}}
\def\IT{{\mathbb{T}}}
\def\IZ{{\mathbb{Z}}}

\def\IIB{{\mathbb{\mathbf B}}}
\def\IIC{{\mathbb{\mathbf C}}}
\def\IIN{{\mathbb{\mathbf N}}}
\def\IIQ{{\mathbb{\mathbf Q}}}
\def\IIR{{\mathbb{\mathbf R}}}
\def\IIZ{{\mathbb{\mathbf Z}}}


\def\ELSE{\quad\mbox{else}\quad}
\def\WITH{\quad\mbox{with}\quad}
\def\FOR{\quad\mbox{for}\quad}
\def\AND{\;\mbox{and}\;}
\def\OR{\;\mbox{or}\;}

\def\To{\longrightarrow}
\def\TO{\Longrightarrow}
\def\In{\subseteq}
\def\sm{\setminus}
\def\Inneq{\In_{\!\!\!\!/}}
\def\dmin{\mathop{\dot{-}}}
\def\splus{\oplus}
\def\SEQ{\triangle}
\def\DIV{\uparrow}
\def\INV{\leftrightarrow}
\def\SET{\Diamond}

\def\kto{\equiv\!\equiv\!>}
\def\kin{\subset\!\subset}
\def\pto{\leadsto}
\def\into{\hookrightarrow}
\def\onto{\to\!\!\!\!\!\to}
\def\prefix{\sqsubseteq}
\def\rel{\leftrightarrow}
\def\mto{\rightrightarrows}

\def\B{{\mathsf{{B}}}}
\def\D{{\mathsf{{D}}}}
\def\G{{\mathsf{{G}}}}
\def\E{{\mathsf{{E}}}}
\def\J{{\mathsf{{J}}}}
\def\K{{\mathsf{{K}}}}
\def\L{{\mathsf{{L}}}}
\def\R{{\mathsf{{R}}}}
\def\T{{\mathsf{{T}}}}
\def\U{{\mathsf{{U}}}}
\def\W{{\mathsf{{W}}}}
\def\Z{{\mathsf{{Z}}}}
\def\w{{\mathsf{{w}}}}
\def\HP{{\mathsf{{H}}}}
\def\C{{\mathsf{{C}}}}
\def\Tot{{\mathsf{{Tot}}}}
\def\Fin{{\mathsf{{Fin}}}}
\def\Cof{{\mathsf{{Cof}}}}
\def\Cor{{\mathsf{{Cor}}}}
\def\Equ{{\mathsf{{Equ}}}}
\def\Com{{\mathsf{{Com}}}}
\def\Inf{{\mathsf{{Inf}}}}

\def\Tr{{\mathrm{Tr}}}
\def\Sierp{{\mathrm Sierpi{\'n}ski}}
\def\psisierp{{\psi^{\mbox{\scriptsize\Sierp}}}}
\def\cl{{\mathrm{{cl}}}}
\def\Haus{{\mathrm{{H}}}}
\def\Ls{{\mathrm{{Ls}}}}
\def\Li{{\mathrm{{Li}}}}

\def\CL{\mathsf{CL}}
\def\ACC{\mathsf{ACC}}
\def\DNC{\mathsf{DNC}}
\def\ATR{\mathsf{ATR}}
\def\LPO{\mathsf{LPO}}
\def\LLPO{\mathsf{LLPO}}
\def\WKL{\mathsf{WKL}}
\def\RCA{\mathsf{RCA}}
\def\ACA{\mathsf{ACA}}
\def\SEP{\mathsf{SEP}}
\def\BCT{\mathsf{BCT}}
\def\IVT{\mathsf{IVT}}
\def\IMT{\mathsf{IMT}}
\def\OMT{\mathsf{OMT}}
\def\CGT{\mathsf{CGT}}
\def\UBT{\mathsf{UBT}}
\def\BWT{\mathsf{BWT}}
\def\HBT{\mathsf{HBT}}
\def\BFT{\mathsf{BFT}}
\def\FPT{\mathsf{FPT}}
\def\WAT{\mathsf{WAT}}
\def\LIN{\mathsf{LIN}}
\def\B{\mathsf{B}}
\def\BF{\mathsf{B_\mathsf{F}}}
\def\BI{\mathsf{B_\mathsf{I}}}
\def\C{\mathsf{C}}
\def\CF{\mathsf{C_\mathsf{F}}}
\def\CN{\mathsf{C_{\IN}}}
\def\CI{\mathsf{C_\mathsf{I}}}
\def\CK{\mathsf{C_\mathsf{K}}}
\def\CA{\mathsf{C_\mathsf{A}}}
\def\WPO{\mathsf{WPO}}
\def\WLPO{\mathsf{WLPO}}
\def\MP{\mathsf{MP}}
\def\BD{\mathsf{BD}}
\def\Fix{\mathsf{Fix}}
\def\Mod{\mathsf{Mod}}

\def\cof{{\mathrm{{cof}}}}

\def\s{\mathrm{s}}
\def\r{\mathrm{r}}
\def\w{\mathsf{w}}

\def\leqm{\mathop{\leq_{\mathrm{m}}}}
\def\equivm{\mathop{\equiv_{\mathrm{m}}}}
\def\leqT{\mathop{\leq_{\mathrm{T}}}}
\def\lT{\mathop{<_{\mathrm{T}}}}
\def\nleqT{\mathop{\not\leq_{\mathrm{T}}}}
\def\equivT{\mathop{\equiv_{\mathrm{T}}}}
\def\nequivT{\mathop{\not\equiv_{\mathrm{T}}}}
\def\leqwtt{\mathop{\leq_{\mathrm{wtt}}}}
\def\equiPT{\mathop{\equiv_{\P\mathrm{T}}}}
\def\leqW{\mathop{\leq_{\mathrm{W}}}}
\def\equivW{\mathop{\equiv_{\mathrm{W}}}}
\def\leqtW{\mathop{\leq_{\mathrm{tW}}}}
\def\leqSW{\mathop{\leq_{\mathrm{sW}}}}
\def\equivSW{\mathop{\equiv_{\mathrm{sW}}}}
\def\leqPW{\mathop{\leq_{\widehat{\mathrm{W}}}}}
\def\equivPW{\mathop{\equiv_{\widehat{\mathrm{W}}}}}
\def\leqFPW{\mathop{\leq_{\mathrm{W}^*}}}
\def\equivFPW{\mathop{\equiv_{\mathrm{W}^*}}}
\def\leqWW{\mathop{\leq_{\overline{\mathrm{W}}}}}
\def\nleqW{\mathop{\not\leq_{\mathrm{W}}}}
\def\nleqSW{\mathop{\not\leq_{\mathrm{sW}}}}
\def\lW{\mathop{<_{\mathrm{W}}}}
\def\lSW{\mathop{<_{\mathrm{sW}}}}
\def\nW{\mathop{|_{\mathrm{W}}}}
\def\nSW{\mathop{|_{\mathrm{sW}}}}
\def\leqt{\mathop{\leq_{\mathrm{t}}}}
\def\equivt{\mathop{\equiv_{\mathrm{t}}}}
\def\leqtop{\mathop{\leq_\mathrm{t}}}
\def\equivtop{\mathop{\equiv_\mathrm{t}}}

\def\bigtimes{\mathop{\mathsf{X}}}

\def\leqm{\mathop{\leq_{\mathrm{m}}}}
\def\equivm{\mathop{\equiv_{\mathrm{m}}}}
\def\leqT{\mathop{\leq_{\mathrm{T}}}}
\def\leqM{\mathop{\leq_{\mathrm{M}}}}
\def\equivT{\mathop{\equiv_{\mathrm{T}}}}
\def\equiPT{\mathop{\equiv_{\P\mathrm{T}}}}
\def\leqW{\mathop{\leq_{\mathrm{W}}}}
\def\equivW{\mathop{\equiv_{\mathrm{W}}}}
\def\nequivW{\mathop{\not\equiv_{\mathrm{W}}}}
\def\leqSW{\mathop{\leq_{\mathrm{sW}}}}
\def\equivSW{\mathop{\equiv_{\mathrm{sW}}}}
\def\leqPW{\mathop{\leq_{\widehat{\mathrm{W}}}}}
\def\equivPW{\mathop{\equiv_{\widehat{\mathrm{W}}}}}
\def\nleqW{\mathop{\not\leq_{\mathrm{W}}}}
\def\nleqSW{\mathop{\not\leq_{\mathrm{sW}}}}
\def\lW{\mathop{<_{\mathrm{W}}}}
\def\lSW{\mathop{<_{\mathrm{sW}}}}
\def\nW{\mathop{|_{\mathrm{W}}}}
\def\nSW{\mathop{|_{\mathrm{sW}}}}

\def\botW{\mathbf{0}}
\def\midW{\mathbf{1}}
\def\topW{\mathbf{\infty}}

\def\pol{{\leq_{\mathrm{pol}}}}
\def\rem{{\mathop{\mathrm{rm}}}}

\def\cc{{\mathrm{c}}}
\def\d{{\,\mathrm{d}}}
\def\e{{\mathrm{e}}}
\def\ii{{\mathrm{i}}}

\def\Cf{C\!f}
\def\id{{\mathrm{id}}}
\def\pr{{\mathrm{pr}}}
\def\inj{{\mathrm{inj}}}
\def\cf{{\mathrm{cf}}}
\def\dom{{\mathrm{dom}}}
\def\range{{\mathrm{range}}}
\def\graph{{\mathrm{graph}}}
\def\Graph{{\mathrm{Graph}}}
\def\epi{{\mathrm{epi}}}
\def\hypo{{\mathrm{hypo}}}
\def\Lim{{\mathrm{Lim}}}
\def\diam{{\mathrm{diam}}}
\def\dist{{\mathrm{dist}}}
\def\supp{{\mathrm{supp}}}
\def\union{{\mathrm{union}}}
\def\fiber{{\mathrm{fiber}}}
\def\ev{{\mathrm{ev}}}
\def\mod{{\mathrm{mod}}}
\def\sat{{\mathrm{sat}}}
\def\seq{{\mathrm{seq}}}
\def\lev{{\mathrm{lev}}}
\def\mind{{\mathrm{mind}}}
\def\arccot{{\mathrm{arccot}}}
\def\cl{{\mathrm{cl}}}

\def\Add{{\mathrm{Add}}}
\def\Mul{{\mathrm{Mul}}}
\def\SMul{{\mathrm{SMul}}}
\def\Neg{{\mathrm{Neg}}}
\def\Inv{{\mathrm{Inv}}}
\def\Ord{{\mathrm{Ord}}}
\def\Sqrt{{\mathrm{Sqrt}}}
\def\Re{{\mathrm{Re}}}
\def\Im{{\mathrm{Im}}}
\def\Sup{{\mathrm{Sup}}}

\def\LSC{{\mathcal LSC}}
\def\USC{{\mathcal USC}}

\def\CE{{\mathcal{E}}}
\def\Pref{{\mathrm{Pref}}}

\def\Baire{\IN^\IN}

\def\TRUE{{\mathrm{TRUE}}}
\def\FALSE{{\mathrm{FALSE}}}

\def\co{{\mathrm{co}}}

\def\BBB{{\tt B}}

\newcommand{\SO}[1]{{{\mathbf\Sigma}^0_{#1}}}
\newcommand{\SI}[1]{{{\mathbf\Sigma}^1_{#1}}}
\newcommand{\PO}[1]{{{\mathbf\Pi}^0_{#1}}}
\newcommand{\PI}[1]{{{\mathbf\Pi}^1_{#1}}}
\newcommand{\DO}[1]{{{\mathbf\Delta}^0_{#1}}}
\newcommand{\DI}[1]{{{\mathbf\Delta}^1_{#1}}}
\newcommand{\sO}[1]{{\Sigma^0_{#1}}}
\newcommand{\sI}[1]{{\Sigma^1_{#1}}}
\newcommand{\pO}[1]{{\Pi^0_{#1}}}
\newcommand{\pI}[1]{{\Pi^1_{#1}}}
\newcommand{\dO}[1]{{{\Delta}^0_{#1}}}
\newcommand{\dI}[1]{{{\Delta}^1_{#1}}}
\newcommand{\sP}[1]{{\Sigma^\P_{#1}}}
\newcommand{\pP}[1]{{\Pi^\P_{#1}}}
\newcommand{\dP}[1]{{{\Delta}^\P_{#1}}}
\newcommand{\sE}[1]{{\Sigma^{-1}_{#1}}}
\newcommand{\pE}[1]{{\Pi^{-1}_{#1}}}
\newcommand{\dE}[1]{{\Delta^{-1}_{#1}}}

\newcommand{\dBar}[1]{{\overline{\overline{#1}}}}

\def\QED{$\hspace*{\fill}\Box$}
\def\rand#1{\marginpar{\rule[-#1 mm]{1mm}{#1mm}}}

\def\BL{\BB}


\newcommand{\bra}[1]{\langle#1|}
\newcommand{\ket}[1]{|#1\rangle}
\newcommand{\braket}[2]{\langle#1|#2\rangle}

\newcommand{\ind}[1]{{\em #1}\index{#1}}
\newcommand{\mathbox}[1]{\[\fbox{\rule[-4mm]{0cm}{1cm}$\quad#1$\quad}\]}


\newenvironment{eqcase}{\left\{\begin{array}{lcl}}{\end{array}\right.}

\theoremstyle{definition}
\newtheorem{theorem}{Theorem}
\newtheorem{definition}[theorem]{Definition}
\newtheorem{problem}[theorem]{Problem}
\newtheorem{assumption}[theorem]{Assumption}
\newtheorem{corollary}[theorem]{Corollary}
\newtheorem{proposition}[theorem]{Proposition}
\newtheorem{lemma}[theorem]{Lemma}
\newtheorem{observation}[theorem]{Observation}
\newtheorem{question}[theorem]{Question}
\newtheorem{example}[theorem]{Example}
\newtheorem{convention}[theorem]{Convention}
\newtheorem{conjecture}[theorem]{Conjecture}
\newtheorem{remark}[theorem]{Remark}

\subjclass{[{\bf Theory of computation}]:  Logic; [{\bf Mathematics of computing}]: Continuous mathematics.}

\begin{abstract}
In computable analysis typically topological spaces with countable bases
are considered. The Theorem of Kreitz-Weihrauch implies that the
subbase representation of a second-countable $\T_0$ space 
is admissible with respect to the topology that the subbase generates.
We consider generalizations of this setting to bases that are
representable, but not necessarily countable. We introduce
the notions of a computable presubbase and a computable prebase.
We prove a generalization of the Theorem of Kreitz-Weihrauch
for the presubbase representation that shows that any such
representation is admissible with respect to the topology
generated by compact intersections of the presubbase elements. 
For computable prebases we obtain representations that are
admissible with respect to the topology
that they generate. These concepts provide a natural way
to investigate many topological spaces that have been studied in computable analysis.
The benefit of this approach is that topologies can be described by
their usual subbases and standard constructions for such subbases can be applied.
Finally we discuss a Galois connection between presubbases and
representations of $\T_0$ spaces that indicates that presubbases
and representations offer particular views on the same mathematical
structure from different perspectives. 
\end{abstract}

\maketitle
\ \\[-0.5cm]\noindent{\footnotesize Version of \today.}\\[-0.5cm]

\section{Introduction}

In the representation based approach to computable analysis,
as initiated by Kreitz and Weihrauch~\cite{KW85,Wei87,Wei00}, 
we work with
{\em representations} $\delta:\In\IN^\IN\to X$, which are surjective partial maps. 
In this situation $(X,\delta)$, or briefly $X$, is called a {\em represented space}.
We say that a function $F:\In\IN^\IN\to\IN^\IN$ {\em realizes} a
multivalued map $f:\In X\mto Y$ on represented spaces $(X,\delta_X)$ and $(Y,\delta_Y)$, in 
symbols, $F\vdash f$, if $\delta_YF(p)\in f\delta_X(p)$ for all $p\in\dom(f\delta_X)$.
Then $f$ is called {\em computable} or {\em continuous}, if it has a realizer of the corresponding type.

On the other hand, every represented space $(X,\delta_X)$ comes naturally equipped with a topology $\OO(X)$,
namely the {\em final topology}
\[\OO(X):=\{U\In X:\delta_X^{-1}(U)\mbox{ open in }\dom(\delta_X)\}\]
induced by the Baire space topology on $\IN^\IN$ via 
the map $\delta_X$ on $X$. 
If not mentioned otherwise, we will assume that $\OO(X)$ is the final topology of the represented space $X$.
The topological spaces whose topology occurs in this way have a special name,
they are called $\qcb$ spaces (quotients of countably based spaces).
These spaces share some special properties~\cite{Sch02}, for instance, they are always {\em sequential}, i.e.,
a subset is closed if and only if it is {\em sequentially closed}, which means that it is closed under limits of
converging sequences.
We also automatically have a 
representation $\delta_{\OO(X)}$ of this topology $\OO(X)$ that is induced by a function space
representation of the space $\CC(X,\IS)$ of the continuous functions $f:X\to\IS$ with 
the {\em Sierpi\'nski space} $\IS=\{0,1\}$. This is because $U\In X$ is open if and only if
its characteristic function $\chi_U:X\to\IS$ is continuous. 
Matthias Schröder has developed a theory of computable topology based on these concepts~\cite{Sch02,Sch02c,Sch21} (see also Pauly~\cite{Pau16} for a concise presentation of some aspects of computable topology). The open sets $U\In X$ that are computable in $\OO(X)$
are called {\em c.e.\ open} and these are exactly the sets with computable characteristic
function $\chi_U:X\to\IS$.

An essential question is when continuity in terms of the representations $\delta_X$ and $\delta_Y$ (as defined above)
coincides for a singlevalued function $f:X\to Y$ with the ordinary notion of continuity with respect
to the topologies $\OO(X)$ and $\OO(Y)$. 
Every function $f$ with a continuous realizer is continuous in the ordinary topological sense, 
but not necessarily the other way around~\cite{Sch02}.
The two concepts of continuity are equivalent if $\delta_X$ and $\delta_Y$ are so-called {\em admissible}
representations with respect to these topologies. The notion of admissibility was originally
defined by Kreitz and Weihrauch just for second-countable $\T_0$ spaces~\cite{KW85} and then later
extended to arbitrary topological spaces by Schr\"oder~\cite{Sch02}. We will adopt the latter definition.

\begin{definition}[Admissibility]
A representation $\delta_X:\In\IN^\IN\to X$ is called {\em admissible} with respect to a
topology $\tau$ on $X$ if the following hold:
\begin{enumerate}
\item $\delta_X$ is continuous with respect to the topology $\tau$ and 
\item any other representation $\delta:\In\IN^\IN\to X$ that is continuous with respect to $\tau$ is also
continuous with respect to the representation $\delta_X$.
\end{enumerate}
A represented space is called {\em admissible} if the representation of the space is admissible with
respect to the final topology $\OO(X)$ induced by it.
\end{definition}

Admissibility of $\delta_X$ means that any continuous representation $\delta$ can be continuously translated into $\delta_X$.  
Hence the names $p\in\IN^\IN$ of $\delta_X$ provide the minimal information on the points $x\in X$
that is required for the representation to be continuous in the topological sense. 
Schr\"oder proved~\cite[Theorem~7, Lemma~8]{Sch02} that whether a representation $\delta_X$ 
is admissible with respect to some topology $\tau$ only depends on the sequentialization of that topology\footnote{The {\em sequentialization} $\seq(\tau)$ of a topology $\tau$ is the smallest topology that contains $\tau$ and is {\em sequential}.}.

\begin{theorem}[Schr\"oder 2002]
\label{thm:Schroder-admissible}
Let $(X,\delta_X)$ be a represented space with final topology $\OO(X)$ and let $\tau$ be some topology on $X$.
Then the following are equivalent:
\begin{enumerate}
\item $\delta_X$ is admissible with respect to $\tau$.
\item $\delta_X$ is admissible with respect to $\seq(\tau)$.
\end{enumerate}
If one these conditions holds, then $\seq(\tau)=\OO(X)$.
\end{theorem}

In particular, this means that if $\delta$ is admissible with respect to some topology, then also with respect to $\OO(X)$.
Schr\"oder also showed~\cite[Theorem~13]{Sch02} that a topology for which an admissible representation exists
is necessarily $\T_0$. 
The topological spaces that admit an admissible representation are exactly those
whose sequentializations are so-called $\qcb_{0}$ spaces 
(the $\T_0$ spaces among the $\qcb$ spaces).
Finally, Schröder also proved that the representation $\delta_{\OO(X)}$ is always admissible 
with respect to the Scott topology on $\OO(X)$~\cite{Sch02c} and that admissibility of a represented
space can be characterized with the help of the {\em neighborhood map} 
$\UU:X\to\OO\OO(X),x\mapsto\{U\in\OO(X):x\in U\}$ of the space as follows. 

\begin{proposition}[Admissibility]
\label{prop:admissibility}
A represented space $X$ is admissible if and only if its neighborhood map $\UU:X\to\OO\OO(X)$ is a continuous embedding
with respect to the involved representations.
\end{proposition}

Being an embedding with respect to the representations means that $\UU$ is injective and $\UU$ as well as
its partial inverse have continuous realizers. Likewise, we define {\em computable embeddings} with computable
realizers. The map $\UU$ is always computable, hence the essential condition is the one that concerns its inverse.
In light of this result one can consider admissibility also as an effectivization of the $\T_0$ property
as the map $\UU$ is injective if and only if the topology $\OO(X)$ is $\T_0$ ($\T_0$ spaces are also known as a {\em Kolmogorov spaces}).
Consequently, we call $X$ a {\em computable Kolmogorov space} if $\UU$ is a computable embedding.\footnote{We 
prefer this terminology over the notion of {\em computable admissibility} that is also used for this concept, as one and the same
topological space can be a computable Kolmogorov space with two representations that are not computably equivalent, whereas topological admissibility characterizes a unique equivalence class of representations with respect to continuous equivalence. 
The notion of admissibility is often linked to this uniqueness property that the computable version does not share.}

Kreitz and Weihrauch~\cite{KW85} have shown that every second-countable $\T_0$ space has an admissible
representation. Namely, if $B=(B_n)_{n\in\IN}$ is a subbase of some topology $\tau$ on $X$, 
then the {\em subbase representation} $\delta^B:\In\IN^\IN\to X$ with
\[\delta^B(p)=x:\iff\range(p)-1=\{n\in\IN:x\in B_n\}\]
for all $p\in\IN^\IN$ and $x\in X$ is admissible with respect to $\tau$. 
Here we do not use $\range(p)$ but $\range(p)-1:=\{n\in\IN:n+1\in\range(p)\}$ in order to allow for the empty set to be enumerated.
We formulate the aforementioned result as a theorem.

\begin{theorem}[Kreitz-Weihrauch 1985]
\label{thm:Kreitz-Weihrauch}
If $X$ is a $\T_0$ space with a countable subbase $B$, then the subbase representation
$\delta^B$ is admissible with respect to the topology of $X$.
\end{theorem}

This theorem does not only guarantee that there are sufficiently many nice representations,
but it also provides a tool that can be used to show that a concrete given representation $\delta_X$ is admissible. One just
needs to prove that the representation is {\em topologically equivalent} to some subbase representation $\delta^B$
with respect to a known subbase $B$ of the topology of interest, i.e., that $\id:X\to X$ is continuous in both directions
with respect to $\delta_X$ and $\delta^B$, respectively. 

One purpose of this article is to provide a similar tool for general (not necessarily second-countable) $\T_0$ spaces.
A second goal is to develop a theory of computable bases for represented spaces very much along the lines
of the second-countable case~\cite{BR26}. Finally, we also demonstrate that these concepts have interesting applications.

\section{The Presubbase Theorem}

We start with the following definition that extends the concept of a subbase representation beyond the countable case.

\begin{definition}[Presubbase]
Let $X$ be a set. We call a family $(B_y)_{y\in Y}$ of subsets of $X$ 
a {\em presubbase} for $X$,
if $Y$ is a represented space and its {\em transpose}
\[B^\T:X\to\OO(Y),x\mapsto\{y\in Y:x\in B_y\}\] 
is well-defined and injective. 
\end{definition}

Injectivity of $B^\T$ implies that $(B_y)_{y\in Y}$ is a subbase of some $\T_0$ topology on $X$. 
We note that every countable subbase $B:\IN\to\OO(X)$ of a $\T_0$ topology is a particular instance of a presubbase,
as $B^\T:X\to\OO(\IN)$ is always well-defined. Hence, the following definition generalizes the concept
of a subbase representation by Kreitz and Weihrauch as it is specified above.

\begin{definition}[Presubbase representation]
Let $(B_y)_{y\in Y}$ be a presubbase of a set $X$. 
We define the {\em presubbase representation} $\delta^B:\In\IN^\IN\to X$ by
\[\delta^B(p)=x:\iff \delta_{\OO(Y)}(p)=\{y\in Y:x\in B_y\}\]
for all $p\in\IN^\IN$ and $x\in X$.
\end{definition}

We note that the map $\delta^B$ is well-defined as $B^\T$ is well-defined and injective.
For short, we can also write $\delta^B=(B^\T)^{-1}\circ\delta_{\OO(Y)}$.
The reason that we speak about a {\em pre}subbase representation in this general situation and not about
a subbase representation is that $\delta^B$ is not necessarily admissible with respect to the
topology generated by $(B_y)_{y\in Y}$. 
However, it is admissible with respect to a closely related 
topology generated by compact intersections of the sets $B_y$, which is our first main result.
By a {\em compact} set $K\In X$ we mean a set with the property that every open cover
has a finite subcover (no Hausdorff condition involved).

\begin{theorem}[Presubbase theorem]
\label{thm:presubbase}
Let $(B_y)_{y\in Y}$ be a presubbase of a set $X$. 
Then $(X,\delta^B)$ is a computable Kolmogorov space and $\delta^B$ is admissible with respect to the
topology $\tau$ on $X$ that is generated by the base sets $\bigcap_{y\in K}B_y$ for every compact $K\In Y$. 
\end{theorem}

We use the convention that $\bigcap_{y\in\varnothing}B_y=X$.
We note that this result generalizes Theorem~\ref{thm:Kreitz-Weihrauch} by Kreitz-Weihrauch
as for countable subbases $B:\IN\to\OO(X)$
the compact subsets $K\In\IN$ are exactly the finite subsets and hence the topology
generated by $\bigcap_{n\in K}B_n$ for compact $K\In\IN$ is exactly the
same topology as the topology generated by the subbase $B$ itself.

In order to prove Theorem~\ref{thm:presubbase} we need to discuss the Scott topology,
which we are not going to define here as the next proposition embodies everything we need to know about it.
Let $X$ be a topological space with some topology $\OO(X)$ that is itself equipped with the Scott topology.
We define $\FF_K:=\{U\in\OO(X):K\In U\}$ for arbitrary subsets $K\In X$.
Then the sets $\FF_K$ are open in $\OO(X)$ for compact $K\In X$
and they form a base of a topology that is sometimes called the {\em compact-open topology} on $\OO(X)$
or also the {\em upper Fell topology} (a notion that is typically applied dually to the space of closed subsets).
In general, this topology is just included in the Scott topology but not identical to it.
Spaces $X$ for which the two topologies coincide are called {\em consonant}.
For sequential spaces $X$ the two topologies share at least the same convergence relation
and in this sense sequential spaces are ``sequentially consonant''.
This result is well-known and we formulate a version with yet another characterization
(see \cite[Proposition~2.2]{Sch15}
for the proof of (1)$\iff$(2)$\iff$(4) and 
\cite[Lemma~4.2.2]{Sch02c} for the proof of (2)$\iff$(3)).

\begin{proposition}[Scott convergence]
	\label{prop:Scott-convergence}
	Let $X$ be a sequential topological space. Let $(U_n)$ be a sequence in $\OO(X)$ and $U\in\OO(X)$.
	Then the following are equivalent:
	\begin{enumerate}
		\item $U_n\to U$ with respect to the Scott topology.
		\item $U_n\to U$ with respect to the topology which is generated by the base of sets $\FF_K$ 
		over all compact $K\In X$.
		\item $U_n\to U$ with respect to the topology which is generated by the subbase of sets $\FF_{\{x,x_n:n\in\IN\}}$
		over all sequences $(x_n)$ and $x$ in $X$ such that $x_n\to x$.
		\item $U\In\bigcup_{k\in\IN}\left(\bigcap_{n\geq k}U_n\right)^\circ$.
	\end{enumerate}
\end{proposition}

Here $A^\circ$ denotes the {\em interior} of the set $A$.
As a consequence of this result, $\OO(X)$ is not just admissibly represented with respect
to the Scott topology, but also with respect to the compact-open topology.
By Theorem~\ref{thm:Schroder-admissible} we 
obtain the following characterization of the topology of the space $\OO(X)$.

\begin{corollary}[Schröder 2002]
\label{cor:Scott}
Let $X$ be a represented space. Then $\OO(X)$ is endowed with the Scott topology
that is the sequentialization of the compact-open topology and the standard
representation of $\OO(X)$ is admissible with respect to both topologies.
\end{corollary}

Now we are actually prepared to prove Theorem~\ref{thm:presubbase}.

\begin{proof}[Proof of Theorem~\ref{thm:presubbase}]
Since $(B_y)_{y\in Y}$ is a presubbase of $X$, the map $B^\T:X\to\OO(Y)$ is well-defined and injective. 
We consider the represented space $(X,\delta^B)$ and we prove that it is a computable Kolmogorov space. 
We represent $\OO(X)$ and $\OO(Y)$ in the usual way, which means that they are endowed
with the respective Scott topologies by Corollary~\ref{cor:Scott}.
The definition $\delta^B=(B^\T)^{-1}\circ\delta_{\OO(Y)}$ ensures that
$B^\T:X\to\OO(Y)$ becomes a computable embedding with respect to $\delta^B$
as a representation of $X$. This implies that $B:Y\to\OO(X)$ is computable too.
We need to prove that also $\UU:X\to\OO\OO(X)$ is a computable embedding.
Given $\UU_x=\{U\in\OO(X):x\in U\}\in\OO\OO(X)$ for some $x\in X$, we can compute
\[B^\T_x=\{y\in Y:x\in B_y\}=B^{-1}(\UU_x)\in\OO(Y)\]
since $B$ is computable. And since $B^\T$ is a computable embedding, we can compute $x\in X$ from this set.
Hence, also $\UU:X\to\OO\OO(X)$ is a computable embedding and $(X,\delta^B)$ is a computable Kolmogorov space.

In particular, $\delta^B$ is admissible with respect to its final topology $\OO(X)$ by Proposition~\ref{prop:admissibility}.
We need to prove that it is also admissible with respect to the topology $\tau$.
We already know by Corollary~\ref{cor:Scott} that $\OO(Y)$ is admissibly represented with
respect to the compact-open topology generated by the sets $\FF_K=\{U\in\OO(Y):K\In U\}$ over all compact $K\In Y$.
Now we note that
\[\bigcap_{y\in K}B_y=\{x\in X:K\In B_x^\T\}=(B^\T)^{-1}(\FF_K)\]
and $(B^\T)^{-1}(\OO(Y))=X$. Hence, $\tau$ is the initial topology of $B^\T$
with respect to the compact-open topology on $\OO(Y)$,
and hence $\delta^B=(B^\T)^{-1}\circ\delta_{\OO(Y)}$ is admissible with respect to $\tau$ by~\cite[Section~4.2]{Sch02} (see also \cite[Proposition~4.1.4]{Sch02c}).
\end{proof}

For simplicity we denote convergent sequences $(x_n)_{n\in\IN}$ that
converge to $x_\infty$  
as $(x_n)_{n\in\IN_\infty}$ with $\IN_\infty:=\IN\cup\{\infty\}$.
By Proposition~\ref{prop:Scott-convergence} we could
alternatively also use the subbase sets
$\FF_K$ with $K=\{y_n:n\in\IN_\infty\}$ over
all convergent sequences $(y_n)_{n\in\IN_\infty}$ in $Y$ in the previous proof.
Hence, we obtain the following.

\begin{remark}
	\label{rem:presubbase}
The statement of Theorem~\ref{thm:presubbase} remains true if we
replace the intersections $\bigcap_{y\in K}B_y$ by intersections $\bigcap_{n\in\IN_\infty} B_{y_n}$
over all convergent sequences $(y_n)_{n\in\IN_\infty}$ in $Y$.
\end{remark}

General so called {\em $Y$-indexed bases} of type $B:Y\to\OO(X)$ have
also been independently considered in \cite[Definition~6.1]{dBSS16}.
The remark above relates our results 
regarding presubbases to the 
so-called {\em sequentially $Y$--indexed
generating systems} considered in~\cite[Definition~7.1]{dBSS16}.

When we have a presubbase $B$ of a represented space $X$, then we have to 
deal in general with at least four different topologies, all generated by $B$ in 
different ways. We illustrate these in Figure~\ref{fig:topologies}.
By $\bigcap_\KK B$ we denote the subbase described in Theorem~\ref{thm:presubbase}
and by $\bigcap_\infty B$ we denote the subbase generated over all convergent
sequences, as described in Remark~\ref{rem:presubbase}.

\begin{figure}[htb]
\begin{footnotesize}
\begin{center}
	\begin{minipage}{2.8cm}
		topology generated\\
		by $B$ as a subbase
	\end{minipage}
	$\In\,\,$
	\begin{minipage}{3.2cm}
	topology generated by\\
	$\bigcap_\infty B$ as a subbase
   \end{minipage}
	$\In\,\,$
   \begin{minipage}{3.2cm}
	topology generated by\\
	$\bigcap_\KK B$ as a subbase
   \end{minipage}
	$\In\,\,$
	\begin{minipage}{3cm}
		final topology of the\\
		representation $\delta^B$
	\end{minipage}
\end{center}
\end{footnotesize}
\caption{Topologies associated to a presubbase $B$ of a represented space.}
\label{fig:topologies}
\end{figure}
	
The fourth (right-hand side) topology in Figure~\ref{fig:topologies} is always the sequentialization of the second and of the third one.
These are consequences of Theorem~\ref{thm:Schroder-admissible}, combined with 
Remark~\ref{rem:presubbase} and Theorem~\ref{thm:presubbase}, respectively.
In the case of countable presubbases $B:\IN\to\OO(X)$ all four topologies
coincide. We introduce a name for presubbases for which the first and the second
topologies coincide.

\begin{definition}[Convergent intersection property]
We say that a presubbase $B:Y\to\OO(X)$ of a represented space $X$ satisfies the
{\em convergent intersection property} if for every convergent sequence $(y_n)_{n\in\IN_\infty}$ in $Y$ there exists a finite set $F\In Y$ with
$\bigcap_{n\in\IN_\infty}B_{y_n}=\bigcap_{y\in F}B_y$.
\end{definition}

For presubbases $B$ with the convergent intersection property all
four topologies in Figure~\ref{fig:topologies} have the fourth one as sequentialization.
	
\begin{corollary}[Convergent intersection property]
\label{cor:convergent-intersection}
Let $(B_y)_{y\in Y}$ be a presubbase of a set $X$ that satisfies the convergent intersection property. 
Then $\delta^B$ is admissible with respect to the
topology $\tau$ on $X$ that is generated by $B$. 
\end{corollary}

\section{Computable Presubbases, Prebases and Bases}

In this section we want to develop a theory of computable bases $B:Y\to\OO(X)$ along the lines of the theory
of computable bases for second-countable spaces~\cite{BR26}.
Bases with more complicated index sets than $Y=\IN$ have already been studied in other contexts, see, e.g., \cite{dBSS16}.
Firstly, the proof of Theorem~\ref{thm:presubbase} already indicates how we can define computable
presubbases.

\begin{definition}[Computable presubbase]
\label{def:computable-presubbase}
Let $X$ and $Y$ be represented spaces. 
Then $B:Y\to \OO(X)$ is called a {\em computable presubbase}
of $X$ if the transpose  
\[B^\T:X\to\OO(Y),x\mapsto\{y\in Y:x\in B_y\}\]
is a well-defined computable embedding.
\end{definition}

We note that $B^\T$ is well-defined and computable if and only if $B$ is so.
Obviously, a presubbase $B$ of $X$ is a computable presubbase of $X$ if and only
if the representation of $X$ is computably equivalent to $\delta^B$.
This observation together with Theorems~\ref{thm:presubbase} and \ref{thm:Schroder-admissible} yield the following corollary.

\begin{corollary}[Presubbases]
\label{cor:presubbase-theorem}
If $X$ is a represented space with a computable presubbase $B:Y\to\OO(X)$, then $X$ is a computable
Kolmogorov space and $\OO(X)=\seq(\tau)$ for the topology $\tau$ that is generated by the base
sets $X$ and $\bigcap_{y\in K}B_y$ for every compact $K\In Y$. 
\end{corollary}

In analogy to the countable case we can now define the concept of a computable base.
In the second-countable case we demand that every {\em finite} intersection of subbase
sets should be computably representable as a {\em countable} union of subbase sets.
If we replace ``finite'' by ``compact'' and ``countable'' by ``overt'', then we obtain the general concept
of a computable base. 
By $\KK_-(Y)$ we denote the set of saturated compact subsets $K\In Y$
represented via continuous maps $\forall_K:\OO(X)\to\IS$ and by $\AA_+(Y)$ we denote
the space of closed subsets $A\In Y$ represented via continuous maps $\exists_A:\OO(X)\to\IS$.
Here $\forall_K$ is just the characteristic function of $\FF_K$ and $\exists_A$ is the
characteristic function of $\TT_A:=\{U\in\OO(X):A\cap U\not=\varnothing\}$ (see~\cite{Pau16,Sch21}). We recall that a set $A\In X$ is called {\em overt}, if $\exists_A$
is computable. 
Using the above terminology we obtain the following straightforward definition.

\begin{definition}[Computable prebase]
\label{def:computable-base}
Let $X$ be a represented space with a computable presubbase $B:Y\to\OO(X)$.
Then $B$ is called a {\em computable prebase} of $X$, if there is a computable 
$R:\KK_-(Y)\mto\AA_+(Y)$ such that
\[\bigcap_{y\in K}B_y=\bigcup_{y\in A}B_y\]
for every $K\in\KK_-(Y)$ and $A\in R(K)$.
We work with the assumption that $\bigcap_{y\in\varnothing}B_y=X$.
We call a computable prebase $B$ a {\em computable base} of $X$ if $B$ is actually a base of $X$.
\end{definition}

Our definition of
a computable base is a computable analogue of the definition of a
$Y$--base provided in \cite[Definition~6.1]{dBSS16}. However, we note that the
more important concept for us is that of a computable prebase, which is 
related to the {\em sequential bases} in \cite[Definition~7.1]{dBSS16}.

From a purely topological perspective, computable prebases have in particular
the property that compact intersections are open in the topology generated 
by the prebase.
Hence, by Corollary~\ref{cor:presubbase-theorem} computable prebases 
characterize the topology of their spaces up to sequentialization.

\begin{corollary}[Computable prebases]
\label{cor:computable-prebase}
If $X$ is a represented space that has a computable prebase that generates a topology $\tau$,
then $\OO(X)=\seq(\tau)$.
\end{corollary}

In the following we will need a couple of computable operations for the spaces $\OO(X)$, $\AA_+(X)$ and $\KK_-(X)$,
respectively. Most of these operations have been considered before (see~\cite{Pau16}).
We recall that a subset $K\In X$ of a topological space $X$ is called {\em saturated}
if $K=\sat(K):=\bigcap\FF_K$. Here $\sat(K)$ is called the {\em saturation} of $K$.

\begin{proposition}[Computable operations on hyperspaces]
\label{prop:computable-hyperspace}
Let $X,Y$ be represented spaces. Then the following hold:
\begin{enumerate}
\item $\UU:X\to\OO\OO(X)$ is computable \hfill (neighborhood map)
\item $\inj:X\to\AA_+(X),x\mapsto\overline{\{x\}}$ is computable \hfill (closed injection)
\item $\inj:X\to\KK_-(X),x\mapsto\sat\{x\}$ is computable \hfill (compact injection)
\item $\KK_-:\CC(X,Y)\to\CC(\KK_-(X),\KK_-(Y)),f\mapsto(K\mapsto\sat f(K))$ is computable \hfill (compact images)
\item $\AA_+:\CC(X,Y)\to\CC(\AA_+(X),\AA_+(Y)),f\mapsto(A\mapsto\overline{f(A)})$ is computable \hfill (closed images)
\item $\sec:X\times\OO(X\times Y)\to\OO(Y),(x,U)\mapsto \{y\in Y:(x,y)\in U\}$ is computable \hfill (section)
\item $\times:\OO(X)\times\OO(Y)\to\OO(X\times Y),(V,U)\mapsto V\times U$ is computable \hfill (product)
\item $\times:\AA_+(X)\times\AA_+(Y)\to\AA_+(X\times Y),(A,B)\mapsto A\times B$ is computable \hfill (product)
\item $\bigcup:\AA_+\OO(X)\to\OO(X),\AA\mapsto\bigcup\AA$ is computable \hfill (overt union of open)
\item $\bigcap:\KK_-\OO(X)\to\OO(X),\KK\mapsto\bigcap\KK$ is computable \hfill (compact intersection of open)
\item $\bigcup:\AA_+\AA_+(X)\to\AA_+(X),\AA\mapsto\overline{\bigcup\AA}$ is computable \hfill (overt union of overt)
\item $\bigcup:\KK_-\KK_-(X)\to\KK_-(X),\KK\mapsto\sat\bigcup\KK$ is computable \hfill 
(compact union of compact)
\item $\FF:\OO(X)\to\KK_-\OO(X),U\mapsto \FF_U$ is a computable embedding \hfill (compact filter)
\item $\FF:\KK_-(X)\to\OO\OO(X),K\mapsto \FF_K$ is a computable embedding \hfill (open filter)
\item $\TT:\AA_+(X)\to\OO\OO(X),A\mapsto \TT_A$ is a computable embedding \hfill (trace)
\item $\Box:\OO(X)\to\OO\KK_-(X),U\mapsto\{K:K\In U\}$ is a computable embedding \hfill (box)
\item $\Diamond:\OO(X)\to\OO\AA_+(X),U\mapsto\{A:A\cap U\not=\varnothing\}$ is a computable embedding \hfill (diamond)
\item $\CC\OO:\CC(X,Y)\to\OO(\KK_-(X)\times\OO(Y)),f\mapsto\{(K,U)\in\KK_-(X)\times\OO(Y):f(K)\In U\}$\\ is computable (a computable embedding if $Y$ is computable Kolmogorov)\hfill (compact-open)\\
\end{enumerate}
For arbitrary sets $\AA,\KK\In\OO(X)$ we have $\bigcup \AA=\bigcup \overline{\AA}$
and $\bigcap\KK=\bigcap\sat(\KK)$.
\end{proposition}
\begin{proof}
	The results (1)--(10) can all be found in \cite{Pau16}, see Propositions~4.2, 5.5, 7.4, Section~9 and Corollaries 10.2 and 10.4 therein. We briefly discuss (11)--(18), which are all easy to see.\\
	(13) We have $\FF_U\In\UU\iff U\in\UU$
	     and hence $\forall_{\FF_U}(\UU)=\chi_\UU(U)$. Hence $\FF$ is computable.
	     We also have $x\in U\iff U\in\UU_x$ and hence $\chi_U(x)=\forall_{\FF_U}(\UU_x)$,
	     which proves that $\FF$ is a computable embedding.\\
	(14) We have $U\in\FF_K\iff K\In U$
	and hence $\chi_{\FF_K}(U)=\forall_K(U)$.\\
	(15) We have $U\in\TT_A\iff A\cap U\not=\varnothing$
	and hence $\chi_{\TT_A}(U)=\exists_A(U)$.\\
	(16) We have $K\in\Box U\iff K\In U\iff U\in\FF_K$. Hence computability of $\FF$ in (14)
	 implies computability of $\Box$. We also have $x\in U\iff\sat\{x\}\In U\iff\sat\{x\}\in\Box U$. Hence
	$\Box$ is a computable embedding as $\inj:X\to\KK_-(X)$ is computable by (3).\\
	(17) We have $A\in\Diamond U\iff A\cap U\not=\varnothing\iff U\in\TT_A$. Hence
	computability of $\Diamond$ follows from (15). 
	We also have $x\in U\iff\overline{\{x\}}\cap U\not=\varnothing\iff\overline{\{x\}}\in\Diamond U$.
	Hence $\Diamond$ is a computable embedding as $\inj:X\to\AA_+(X),x\mapsto\overline{\{x\}}$ is computable by (2).\\
	(11) We have $\overline{\bigcup\AA}\cap U\not=\varnothing\iff\bigcup\AA\cap U\not=\varnothing\iff(\exists A\in\AA)\;A\in\Diamond U\iff\exists_\AA(\Diamond U)=1$.
	Hence (11) follows from (17).\\
	(12) We have $\sat\bigcup\KK\In U\iff\bigcup\KK\In U\iff(\forall K\in\KK)\;K\in\Box U\iff\forall_\KK(\Box U)=1$. Hence, (12) follows from (16).\\
	(18) Since $f(K)\In U\iff U\in\FF_{\sat f(K)}$, computability of $\CC\OO$ follows from (14) and (4).
	If $Y$ is a computable Kolmogorov space, then $\UU:Y\to\OO\OO(Y)$ is a computable embedding.
	Since $U\in\UU_{f(x)}\iff f(x)\in U\iff\sat f\{x\}\In U$, it follows that $\CC\OO$ is a computable
	embedding in this case too.\\
	For arbitrary $\AA,\KK\In\OO(X)$ we have
	\begin{itemize}
	\item $x\in\bigcup\overline{\AA}\iff\overline{\AA}\cap\UU_x\not=\varnothing
	\iff\AA\cap\UU_x\not=\varnothing\iff x\in\bigcup\AA$ and
	\item $x\in\bigcap\sat\KK\iff\sat\KK\In\UU_x\iff\KK\In\UU_x\iff x\in\bigcap\KK$,
	\end{itemize}
	which proves the additional remark.
	\end{proof}

Often, we can establish that a computable presubbase is even a computable prebase
because it is already closed under compact intersections.

\begin{lemma}[Closure under compact intersections]
\label{lem:closure-compact-intersection}
Let $X$ be a represented space with a computable presubbase $B:Y\to\OO(X)$.
If there is a computable ${R:\KK_-(Y)\mto Y}$ such that 
$\bigcap_{y\in K}B_y=B_{z}$ for every $z\in R(K)$, then $B$ is a computable prebase.
\end{lemma}
\begin{proof}
This follows as $\inj:Y\to\AA_+(Y),z\mapsto\overline{\{z\}}$ is computable and
$\bigcup_{y\in\overline{\{z\}}}B_y=B_z$ by Proposition~\ref{prop:computable-hyperspace}.
\end{proof}

Every computable presubbase can be converted into a computable prebase just by
taking the closure under compact intersections. 

\begin{proposition}[Computable prebases]
\label{prop:computable-prebase}
Let $X$ be a represented space with a computable presubbase $B:Y\to\OO(X)$.
Then
\[\cap_\KK B:\KK_-(Y)\to\OO(X),K\mapsto\bigcap_{y\in K}B_y,\]
is a computable prebase with the understanding that the empty intersection is $X$.
\end{proposition}
\begin{proof}
We use Proposition~\ref{prop:computable-hyperspace}.
Firstly, we note that $\cap_\KK B(\varnothing)=X$.
Next, we need to prove that $\cap_\KK B$ is a computable presubbase, i.e.,
that 
\[(\cap_\KK B)^\T:X\to\OO\KK_-(Y),x\mapsto\{K\in\KK_-(Y):x\in\cap_{y\in K}B_y\}\]
is well-defined and a computable embedding.
Since $B^\T$ is a computable embedding, it follows 
that $B$ is computable and hence so is $\cap_\KK B$.
Hence $(\cap_\KK B)^\T$ also is well-defined and computable.
It remains to prove that it is a computable embedding.
The map $\inj:Y\to\KK_-(Y),y\mapsto\sat\{y\}$ is computable.
Since $y\in B^\T_x\iff x\in B_y\iff\sat\{y\}\in(\cap_\KK B)^\T_x$,
we obtain $B^\T_x=\inj^{-1}((\bigcap_\KK B)^\T_x)$.
As $B^\T$ is a computable embedding, this implies that $(\cap_\KK B)^\T$ is a computable embedding too.

Finally, we need to prove that $\cap_\KK B$ is a computable prebase,
i.e., that all compact intersections of it can be obtained computably as overt unions.
To this end, we prove that $\cap_\KK\cap_\KK B:\KK_-\KK_-(Y)\to\OO(X)$ is computable.
Since $\bigcup:\KK_-\KK_-(Y)\to\KK_-(Y),\KK\mapsto\sat\bigcup\KK$ is computable and
\[\bigcap_{K\in\KK}\bigcap_{y\in K}B_y=\bigcap_{y\in\bigcup\KK}B_y=\bigcap_{y\in\sat(\bigcup\KK)}B_y,\]
the claim follows with Lemma~\ref{lem:closure-compact-intersection}.
\end{proof}

In fact, every computable Kolmogorov space has a computable base, namely simply the identity. This corresponds to the fact in classical topology that every topology is a base of itself. We recall that according to our definition a computable base is more
than just a base that is computable.

\begin{proposition}[Identity as a base]
	\label{prop:identity-base}
	Let $X$ be a represented space. Then the identity $\id:\OO(X)\to\OO(X)$ is a computable base/prebase/presubbase of $X$
	if and only if $X$ is a computable Kolmogorov space.
\end{proposition}
\begin{proof}
	It follows from $\id^\T=\UU:X\to\OO\OO(X)$ that $\id:\OO(X)\to\OO(X)$ is a computable presubbase if and only 
	if $X$ is a computable Kolmogorov space. If $\id$ is a computable presubbase,
	then it is automatically a computable base by Lemma~\ref{lem:closure-compact-intersection}
	and Proposition~\ref{prop:computable-hyperspace}.
\end{proof}

Altogether, we obtain the following characterization of computable Kolmogorov spaces in terms of their bases.

\begin{theorem}[Computable Kolmogorov spaces and bases]
	\label{thm:Kolmogorov-bases}
	Let $X$ be a represented space. Then the following are pairwise equivalent:
	\begin{enumerate}
		\item $X$ is a computable Kolmogorov space.
		\item $X$ has a computable presubbase.
		\item $X$ has a computable prebase.
		\item $X$ has a computable base.
		\item $\id:\OO(X)\to\OO(X)$ is a computable base of $X$.
	\end{enumerate}
\end{theorem}
\begin{proof}
	Proposition~\ref{prop:identity-base} shows that (1)$\iff$(5).
	It is also clear that we obtain the implications (5)$\TO$(4)$\TO$(3)$\TO$(2). 
	Finally, (2)$\TO$(1) follows from Corollary~\ref{cor:presubbase-theorem}.
\end{proof}

By a {\em represented $\T_0$ space} we mean a represented space $X$
whose final topology is $\T_0$.
We note that another consequence of the above results is that every
represented $\T_0$ space $X$ can be converted into a computable Kolmogorov
space. This is because $\id:\OO(X)\to\OO(X)$ is always a presubbase.
In fact, this conversion preserves the topology and the equivalence class of its representation. For two representations $\delta_1$ of a set $X_1$ and 
$\delta_2$ of a set $X_2$ we write $\delta_1\leq\delta_2$ if $X_1\In X_2$
and the representation $\delta_1$ can be computably reduced to $\delta_2$, i.e., 
if the identity $\id:X_1\to X_2$ is computable.

\begin{theorem}[Represented spaces as computable Kolmogorov spaces]
	\label{thm:represented-Kolmogorov}
	Let $X$ be a represented space with a $\T_0$ topology $\OO(X)$.
	Then we can endow $X$ with another representation that
	turns $X$ into a computable Kolmogorov space without changing
	the topology $\OO(X)$ (and without changing the computable equivalence class of the representation of $\OO(X)$).
\end{theorem}
\begin{proof}
	Let $(X,\delta)$ be a represented space with $\T_0$ topology $\OO(X)$
	and canonical representation $\delta^\circ$ of $\OO(X)$ and $\delta^{\circ\circ}$ of $\OO\OO(X)$.
	Then $\id:\OO(X)\to\OO(X)$ is a presubbase of this space, 
	as $\id^\T=\UU:X\to\OO\OO(X)$ is well-defined and injective.
	Hence, the presubbase representation $\delta^{\bullet}=\UU^{-1}\circ\delta^{\circ\circ}$
	turns $X$ into a computable Kolmogorov space $X^\bullet$ by the Presubbase Theorem~\ref{thm:presubbase}. Let $\delta^{\bullet\circ}$ denote the 
	representation of the open sets $\OO(X^\bullet)$ induced by $\delta^\bullet$. 
	We need to prove that $\OO(X)=\OO(X^\bullet)$ and that $\delta^\circ\equiv\delta^{\bullet\circ}$.
	Since $\UU$ is computable, it is clear that $\delta\leq\delta^\bullet$ and 
	hence $\delta^{\bullet\circ}\leq\delta^\circ$ and $\OO(X^\bullet)\In\OO(X)$. We need to prove the inverse reduction.
	If $x=\delta^\bullet(p)$ and $U=\delta^\circ(q)$ are given by $p,q\in\IN^\IN$, 
	then we have
	$x\in U\iff U\in\UU_x\iff\delta^{\circ}(q)\in\delta^{\circ\circ}(p)$ which can be
	confirmed with the help of $p,q\in\IN^\IN$. This shows $\delta^\circ\leq\delta^{\bullet\circ}$ and $\OO(X)\In\OO(X^\bullet)$.
\end{proof}

In fact, $\delta\mapsto\delta^\bullet$ is a closure operator in the lattice of 
representations of $\T_0$ spaces that was introduced and studied by Schröder~\cite{Sch02c}. 
It converts every representation $\delta$ with final topology $\OO(X)$ into a representation $\delta^\bullet$ 
that is admissible with respect to $\OO(X)$ (by the Presubbase Theorem~\ref{thm:presubbase}).
We will get back to discussing this closure operator in the Epilogue.

\section{Special Bases and Spaces}

In this section we want to discuss some special types of bases and
characterizations of special types of spaces by bases.
The remainder of this article is independent of the results of this section.
Firstly, we introduce two notions of bases that  
were considered in the countable case in~\cite{BR26}.
A {\em Lacombe base} reflects the
fact that open sets can be written as unions of base sets.
A {\em Nogina base} reflects
the fact that for open sets and points therein we can find base sets in between. 
We call a map $f:X\to Y$ {\em computably surjective} if it is computable, surjective
and has a multivalued computable inverse $F:Y\mto X$.

\begin{definition}[Lacombe and Nogina bases]
Let $X,Y$ be represented spaces with a map $B:Y\to\OO(X)$. Then
\begin{enumerate}
	\item $B$ is called
	a {\em computable Lacombe base} of $X$ if 
	\[\bigcup B:\AA_+(Y)\to\OO(X),A\mapsto\bigcup_{y\in A}B_y\]
	is computably surjective. We use $\bigcup_{y\in\varnothing}B_y=\varnothing$.
	\item $B$ is called a {\em computable Nogina base} of $X$ if
	\[N:\In X\times\OO(X)\mto Y,(x,U)\mapsto\{y\in Y:x\in B_y\In U\}\]
	with $\dom(N):=\{(x,U):x\in U\}$ is well-defined and computable.
\end{enumerate}
\end{definition}

A definition similar to that of a Lacombe base has also been considered in a constructive
setting before~\cite[Definition~6.3]{BL12}.
It is easy to see that the mere existence of computable Lacombe or Nogina bases is not a
particularly interesting property as every represented space has such bases.

\begin{proposition}[Identity as Lacombe and Nogina base]
	\label{prop:identity-base-Lacombe-Nogina}
	Let $X$ be a represented space. Then the identity $\id:\OO(X)\to\OO(X)$ is a computable Lacombe base and a computable Nogina base of $X$.
\end{proposition}
\begin{proof}
	It is obvious that $\id:\OO(X)\to\OO(X)$ is a computable Nogina base
	and it is easy to see that it also is a computable Lacombe base.
	This follows from Proposition~\ref{prop:computable-hyperspace} since
	$\inj:\OO(X)\to\AA_+\OO(X),U\mapsto\overline{\{U\}}$ is computable
	and $\bigcup\overline{\{U\}}=U$.
\end{proof}

It is more interesting to ask how different base properties for a given $B:Y\to\OO(X)$
are logically related to each other. We will prove a number of implications,
some of which only hold for specific types of spaces $Y$ or $X$.
It is a consequence of Proposition~\ref{prop:computable-hyperspace}
that every computable Lacombe base is a computable base,
provided the space is a computable Kolmogorov space.

\begin{proposition}[Computable Lacombe bases]
\label{prop:general-computable-base}
Let $X$ be a computable Kolmogorov space.
Then every computable Lacombe base of $X$ is also a computable base.
\end{proposition}
\begin{proof}
We use Proposition~\ref{prop:computable-hyperspace}.
If $\bigcup B:\AA_+(Y)\to\OO(X)$ is computable, then so is the map $B:Y\to\OO(X)$.
This is because $\inj:Y\to\AA_+(Y),y\mapsto\overline{\{y\}}$ is computable
and $\bigcup_{z\in\overline{\{y\}}}B_z=B_y$.
We also have that $\bigcap:\KK_-\OO(X)\to\OO(X)$ is computable 
and $\bigcup B:\AA_+(Y)\to\OO(X)$ has a computable right inverse $S:\OO(X)\mto\AA_+(Y)$.
Hence, given a set $K\in\KK_-(Y)$ we can compute $\KK:=\sat B(K)=\sat\{B_y:y\in K\}\in\KK_-\OO(X)$
and hence $\bigcap\KK=\bigcap_{y\in K}B_y\in\OO(X)$.
Using $S$ we obtain $A\in\AA_+(Y)$ with $\bigcap_{y\in K}B_y=\bigcup_{y\in A}B_y$.
Altogether this describes a computation of a multivalued $R:\KK_-(Y)\mto\AA_+(Y)$ as it is
required for a computable prebase. Because $\bigcup B:\AA_+(Y)\to\OO(X)$ is surjective,  $B$ is actually a base of $X$.
It remains to prove that $B:Y\to\OO(X)$ is a computable presubbase, i.e., that the transpose
$B^\T:X\to\OO(Y)$ is injective and has a computable right inverse. If the set $V=\{y\in Y:x\in B_y\}\in\OO(Y)$
is given, then we can also compute the set $\UU_x=\{U\in\OO(X):x\in U\}\in\OO\OO(X)$.
This is because given $U\in\OO(X)$ we can compute $A\in\AA_+(Y)$ with $U=\bigcup_{y\in A}B_y$
and hence $x\in U\iff(\exists y\in A)\;x\in B_y\iff A\cap V\not=\varnothing$.
But given $\UU_x\in\OO\OO(X)$ we can compute $x\in X$, as $X$ is a computable Kolmogorov space.
This completes the proof that $B:Y\to\OO(X)$ is a computable base of $X$.
\end{proof}

This result generalizes~\cite[Proposition~4.4]{BR26} for the second-countable case.
We call a represented space {\em computably second-countable}, if it
has a computable base $B:\IN\to X$. 
We recall that a representation $\delta$ of $X$ is called {\em computably open}
if $\OO(\IN^\IN)\to\OO(X),U\mapsto\delta(U)$ is well-defined and computable.
The following characterization
is well-known~\cite[Theorem~6.1]{BR26}.

\begin{proposition}[Computable second-countability]
\label{prop:computable-second-countable}
Let $X$ be a computable Kolmogorov space.
Then $X$ is computably second-countable if and only if $X$ admits a computably 
equivalent computably open representation.
\end{proposition}

For such spaces every computable base $B:\IN\to X$ 
is even a computable Lacombe base~\cite[Lemma~3.4]{BR26}. 
We can expand the equivalence to a larger
class of index spaces $Y$.

As a preparation, we need to study the representation
$\delta_{\OO(Y)}$ of open subsets of $Y$.
If $q\in\IN^\IN$, then we write for short $n\in q:\iff n\in\range(q)-1\iff(\exists i)\;q(i)=n+1$,
i.e., if $n$ is ``enumerated'' into $q$ (where $0$ serves as a dummy symbol). 
We use an analogous notation $i\in w$ for words $w\in\IN^*$.
Let $\w_i$ be the $i$--th word according
to some standard numbering of $\IN^+=\IN^*\setminus\{\varepsilon\}$, 
the set of words without the empty word $\varepsilon$. 

\begin{proposition}[Synthetic versus final representation]
\label{prop:synthetic-final}
Let $(Y,\delta_Y)$ be a represented space.
If we define a representation $\delta_{\OO(Y)}$ of $\OO(Y)$ by
\[\delta_{\OO(Y)}(q)=V:\iff\delta_Y^{-1}(V)=\dom(\delta_Y)\cap\bigcup_{j\in q}\w_j\IN^\IN\]
for all $q\in\IN^\IN$ and $V\in\OO(Y)$, then 
$\chi:\OO(Y)\to\CC(Y,\IS),U\mapsto \chi_U$
is a computable isomorphism and hence $\delta_{\OO(Y)}$ is computably
equivalent to the synthetic representation of $\OO(Y)$, defined via continuous characteristic
functions $\chi_U:X\to\IS$.
\end{proposition}
\begin{proof}
Given $U\in\OO(Y)$ by some $\delta_{\OO(Y)}$--name $q$ and $y\in Y$ by some
$\delta_Y$--name $p$, we can easily evaluate $\chi_U(y)$: we just search for some $j\in q$
with $\w_j\prefix p$ and if we find one, then we output $1$, otherwise $0$.
By currying, this proves that $\chi$ is computable.
Now given $\chi_U\in\CC(Y,\IS)$, we have a word function $f:\IN^*\to\IN^*$ that
approximates a realizer of $\chi_U$. We just need to search for words $w$ such that
$f(w)$ contains a $1$ and enumerate corresponding $j$ with $\w_j=w$ into $q$.
If it happens that $w$ is the empty word, then we list instead
all words $w$ of length $1$ in $q$.
This proves that the inverse of $\chi$ is computable.
\end{proof}

We note that the representation $\delta_{\OO(Y)}$ 
is not necessarily total, because the enumeration
of certain words $\w_j$ in a name $q$ enforces the enumeration
of other words $\w_k$ in the enumeration $q$
due to extensionality.
However, every word $w\in\IN^*$ can be extended to a name of $Y$ and hence
$w\IN^\IN\cap\dom(\delta_{\OO(Y)})\not=\varnothing$.
We now investigate this representation on cylinders.

\begin{proposition}[Cylinder and filter]
	\label{prop:cylinder-filter}
	Let $(Y,\delta_Y)$ be a represented space and let $\delta_{\OO(Y)}$ be defined as 
	in Proposition~\ref{prop:synthetic-final}.
	Then we obtain 
	\[\delta_{\OO(Y)}(w\IN^\IN)=\FF_{P_w}\text{ with }P_w:=\bigcup_{j\in w}\delta_Y(\w_j\IN^\IN)\]
	for each $w\in\IN^*$, where $\FF_{P_w}=\{U\in\OO(Y):P_w\In U\}$.
\end{proposition}
\begin{proof}
	For $q\in\dom(\delta_{\OO(Y)})$ we obtain
	$V:=\delta_{\OO(Y)}(q)=\bigcup_{j\in q}\delta_Y(\w_j\IN^\IN)$.
	Hence
	\begin{eqnarray*}
		\delta_{\OO(Y)}(w\IN^\IN)
		&=& \{\delta_{\OO(Y)}(q):q\in w\IN^\IN\cap\dom(\delta_{\OO(Y)})\}\\
		&=& \left\{\bigcup_{j\in q}\delta_Y(\w_j\IN^\IN):q\in w\IN^\IN\cap\dom(\delta_{\OO(Y)})\right\}\\
		&=& \left\{\bigcup_{j\in w}\delta_Y(\w_j\IN^\IN)\cup\bigcup_{j\in r}\delta_Y(\w_j\IN^\IN):wr\in w\IN^\IN\cap\dom(\delta_{\OO(Y)})\right\}\\
		&=& \{V\in\OO(Y):P_w\In V\}\\
		&=& \FF_{P_w}.
	\end{eqnarray*}
	Here ``$\supseteq$'' in the second to last step holds, because if $r\in\dom(\delta_{\OO(Y)})$, $w\in\IN^*$ and $V:=\delta_{\OO(Y)}(r)\supseteq P_w$, then $wr\in\dom(\delta_{\OO(Y)})$ and $\delta_{\OO(Y)}(wr)=V$.
\end{proof}

In order to make the dual behavior of a representation and its induced
representation of open subsets formally precise, we define the following notion of a {\em computably locally compact}
representation. We warn the reader that this property does neither imply local compactness of $Y$ nor of $\dom(\delta_Y)$.
It really just expresses a computable compactness property of the map in a local sense.

\begin{definition}[Computably locally compact representations]
We call a representation $\delta_Y:\In\IN^\IN\to Y$ of $Y$ {\em computably locally compact}
if 
\[\lambda:\IN^+\to\KK_-(Y),w\mapsto\delta_Y(w\IN^\IN)\] 
is well-defined and computable.
\end{definition}

It is easy to see that $\delta_Y$ is  computably open
if $\OO(\IN^\IN)\to\OO(Y),U\mapsto\delta_Y(U)$ is well-defined and computable, or equivalently,
if $\IN^+\to\OO(Y),w\mapsto\delta_Y(w\IN^\IN)$ is well-defined and computable.
Now we can formulate the following duality, where we use the dual computability
properties of the filter maps $U\mapsto\FF_U$ and $K\mapsto\FF_K$ from
Proposition~\ref{prop:computable-hyperspace} together with the
fact that finite unions can be computed for open and compact sets.

\begin{corollary}[Duality]
\label{cor:duality}
Let $Y$ be a represented space. Then the following hold:
\begin{enumerate}
	\item If $Y$ has a computably equivalent computably open representation,
	      then $\OO(Y)$ has a computably equivalent computably locally compact representation.
	\item If $Y$ has a computably equivalent computably locally compact representation,
	      then $\OO(Y)$ has a computably equivalent computably open representation.
\end{enumerate}
\end{corollary}

We note that this duality is related to some purely topological
characterizations in~\cite{Sch04,Sch16}.
Now we want to discuss implications of these observations for bases.

\begin{proposition}[Normal form]
	\label{prop:normal-form}
	Let $(X,\delta_X)$ and $(Y,\delta_Y)$ be represented spaces with
	a presubbase $B:Y\to\OO(X)$ and $\delta_{\OO(X)}, \delta_{\OO(Y)}$ defined
	as in Proposition~\ref{prop:synthetic-final},
	where $\delta_X=\delta^B$ is the presubbase representation of $B$.
	Then we obtain
	\[
	U = \bigcup_{i\in p}\bigcap_{y\in K_{\w_i}}B_y.
	\]
	for each $p\in\dom(\delta_{\OO(X)})$ with $U:=\delta_{\OO(X)}(p)$ and sets
	$K_w:=\bigcup_{j\in w}\delta_Y(\w_j\IN^\IN)$ for $w\in\IN^*$.
	In particular, if $\delta_Y$ is computably locally compact and $B$ is a computable base, then $B$ is a Lacombe base.
\end{proposition}
\begin{proof}
	Let $\delta_{\OO(X)}(p)=U$. With $\delta_X=(B^\T)^{-1}\circ\delta_{\OO(Y)}$ we obtain
	with Proposition~\ref{prop:cylinder-filter}
	\[
	U 
	= \delta_X\left(\bigcup_{i\in p}\w_i\IN^\IN\right)
	= \bigcup_{i\in p} (B^\T)^{-1}\circ\delta_{\OO(Y)}(\w_i\IN^\IN)
	= \bigcup_{i\in p}(B^\T)^{-1}(\FF_{K_{\w_i}})
	= \bigcup_{i\in p}\bigcap_{y\in K_{\w_i}}B_y.
	\]
	If $\delta_Y$ is computably locally compact, then we can compute the sets
	$K_{\w_i}\in\KK_-(Y)$ and if $B$ is a computable base, then we can compute
	$A_i\in\KK_-(Y)$ uniformly in $i$ such that $\bigcup_{y\in A_i}B_y=\bigcap_{y\in K_{\w_i}}B_y$.
	Finally, from the sets $A_i$ we can compute the union $A:=\overline{\bigcup_{i\in p}A_i}$ in $\AA_+(Y)$
	by Proposition~\ref{prop:computable-hyperspace}
	and we obtain 
	\[\bigcup_{y\in A}B_y=\bigcup_{y\in \bigcup_{i\in p}A_i}B_y=\bigcup_{i\in p}\bigcup_{y\in A_i}B_y=\bigcup_{i\in p}\bigcap_{y\in K_{\w_i}}B_y=U.\]
	This proves that $B$ is a computable Lacombe base.
\end{proof}

This result allows us to formulate a computable base theorem
for a class of bases with more general index space $Y$. 
We say that $Y$ is {\em computably locally compactly represented}
if $Y$ has a computably equivalent representation that is computably locally compact.
Some obvious examples of such spaces are $\IN$, $\IN\times2^\IN$ and $\IR^n$. 
However, we emphasize again that computably locally compactly represented spaces are not necessarily
locally compact themselves. Together with Propositions~\ref{prop:general-computable-base} and \ref{prop:normal-form}
we obtain the following.

\begin{theorem}[Base theorem]
\label{thm:base}
	Let $Y$ be computably locally compactly represented space
	and $X$ a computable Kolmogorov space with a map $B:Y\to\OO(X)$.
	Then the following are equivalent:
	\begin{enumerate}
		\item $B:Y\to\OO(X)$ is a computable base of $X$.
		\item $B:Y\to\OO(X)$ is a computable Lacombe base of $X$.
	\end{enumerate}
\end{theorem}

However, this result looks perhaps more useful than it is.
While we have extended the base theorem from the case of $Y=\IN$ to
a larger class of bases $B:Y\to\OO(X)$ with more general index spaces $Y$, 
we have not extended it to a larger class of spaces $X$.
This is because of the following result.

\begin{theorem}[Computable second-countability]
Let $X$ be a computable Kolmogorov space. Then the following are equivalent:
\begin{enumerate}
	\item $X$ is computably second-countable.
	\item $X$ admits a computable base $B:\IN\to\OO(X)$.
	\item $X$ admits a computable base $B:Y\to\OO(X)$ with a computably locally compactly represented space $Y$.
	\item $\OO(X)$ is computably locally compactly represented.
\end{enumerate}
\end{theorem}
\begin{proof}
The condition (2) is our definition of (1). Clearly, (2) implies (3),
as $\IN$ is computably locally compactly represented.
Also (3) implies (1):
If $Y$ is computably locally compactly represented, and $B:Y\to\OO(X)$ is 
a computable base of $X$, then $B^\T:X\to\OO(Y)$ is a computable embedding.
By Corollary~\ref{cor:duality} the space $\OO(Y)$ has a computably equivalent
computably open representation.
This implies that $\OO(Y)$ is computably second-countable by Proposition~\ref{prop:computable-second-countable} ($\OO(Y)$ is 
always a computable Kolmogorov space).
Since $X$ can be computably embedded into $\OO(Y)$, it is also computably second-countable.
Finally, (1) implies (4) by Corollary~\ref{cor:duality}
and Proposition~\ref{prop:computable-second-countable},
and (4) implies (3) by Proposition~\ref{prop:identity-base} .
\end{proof}

This theorem characterizes the scope of spaces that have computable bases
with computably locally compactly represented index space $Y$.
On the first sight, it seems to make Theorem~\ref{thm:base} less useful.
On the other hand, we will see in Example~\ref{ex:cofinite} that even for computably
second-countable spaces $X$ Theorem~\ref{thm:base} does not
hold for arbitrary index spaces $Y$.

Nogina bases seem to be less suitable for a general treatment of computable
topology, as they refer to points, which might not be easily accessible. 
However, if the index space $Y$ is sufficiently
nice, then every Lacombe base is a Nogina base. If the space $X$ itself is
nice, then every Nogina base yields a Lacombe base in a canonical way.

In order to express this canonical representation, we denote 
for every represented space $(X,\delta)$ by $(X_\bot,\delta_\bot)$ 
a version of the space to which we attach a bottom element $\bot\not\in X$.
That is $X_\bot:=X\cup\{\bot\}$ and $\delta_\bot$ is defined by
\[\delta_\bot(\widehat{0}):=\bot\text{ and }\delta_\bot(0^n1p):=\delta(p)\]
for all $p\in\dom(\delta)$ and $n\in\IN$.
Here $\widehat{0}=000...\in\IN^\IN$ denotes the constant zero sequence.
In other words, the element $\bot$ is attached to $X$ such that $\id:X\to X_\bot$
is a computable embedding and $X$ is a c.e.\ open subspace of $X_\bot$.
This yields the least possible one-point compactification of $X$, and its topology
is sometimes called the {\em open extension topology}. For instance, $\{1\}_\bot$
is computably isomorphic to $\IS$.

For every map $B:Y\to\OO(X)$ we define $B_\bot:Y_\bot\to\OO(X)$ as 
extension of $B$ with $B(\bot)=\varnothing$.
It is clear that $B_\bot$ is computable if $B$ is computable, as 
we can map the input to a prefix of a name of $\varnothing$ until it is
clear that the input is different from $\bot$.

\begin{figure}[htb]
	\begin{center}
		\begin{tikzpicture}[scale=0.9,auto=left,thick,every node/.style={fill=blue!00}]
			\node (Base) at   (-3,0) {computable base};
			\node (Nogina) at (3,0) {computable Nogina base};
			\node (Lacombe) at (0,1.5) {computable Lacombe base};    
			\draw[->,dashed] (Lacombe) edge node [right=0.8cm] {\begin{minipage}{3cm}\tiny $Y$ computably\\ quasi-Polish\end{minipage}} (Nogina);
			\draw[->] (Lacombe) edge [bend angle=7, bend left]  (Base);	
			\draw[->,dashed] (Base) edge [bend angle=7, bend left] node [left=0.5cm] {\begin{minipage}{3cm}\tiny $Y$ computably locally\\ compactly represented\end{minipage}} (Lacombe);		
		\end{tikzpicture}
		\caption{Bases $B:Y\to\OO(X)$ for computable Kolmogorov spaces $X$.}
		\label{fig:bases}
	\end{center}
\end{figure}

We say that a represented space $X$ is {\em computably complete}, if its representation
is computably equivalent to a total representation $\delta:\IN^\IN\to X$.
And a represented space $X$ is called a {\em computable quasi-Polish space},
if it is complete and is computable second-countable (i.e., admits a computable basis $B:\IN\to X$).\footnote{Some
authors might consider the empty space as computably complete, but we would like to exclude this special case.}

\begin{proposition}[Nogina and Lacombe bases]
	Let $X$ and $Y$ be represented spaces. 
	\begin{enumerate}
		\item If $B:Y\to\OO(X)$ is a computable Lacombe base of $X$ and $Y$ is a 
		computable quasi-Polish space,
		then $B$ is a computable Nogina base of $X$.
		\item If $B:Y\to\OO(X)$ is a computable Nogina base of $X$ and $X$ is a complete space, then $B_\bot:Y_\bot\to\OO(X)$ is a computable Lacombe base of $X$.
	\end{enumerate}
\end{proposition}
\begin{proof}
	(1) Let $B:Y\to\OO(X)$ be a computable Lacombe base of $X$. Then 
	given some $U\in\OO(X)$ we can compute
	some $A\in\AA_+(Y)$ such that $U=\bigcup_{y\in A}B_y$.
	Since $Y$ is a computable quasi-Polish space, we can compute some
	sequence $(y_n)_n$ in $Y$ with $A=\overline{\{y_n:n\in\IN\}}$ by~\cite[Corollary~21]{BPS20},
	which implies $U=\bigcup_{n\in\IN}B_{y_n}$ (using the statement at the bottom of Proposition~\ref{prop:computable-hyperspace}).
	Now given some $x\in X$ with $x\in U$, we can compute some $n\in\IN$
	with $x\in B_{y_n}\In U$. Hence, $B$ is a computable Nogina base of $X$.\\
	(2) Let $B:Y\to\OO(X)$ be a computable Nogina base and let $X$ be complete.
	We can extend the Nogina operation $N$ to a total computable operation
	\[N:X\times\OO(X)\mto Y_\bot,(x,U)\mapsto
	\left\{\begin{array}{ll}
		\{y\in Y:x\in B_y\In U\} & \text{if $x\in U$}\\
		\bot & \text{otherwise}
	\end{array}\right..
	\]
	Hence, there is a computable function $F:\In\IN^\IN\times\IN^\IN\to Y_\bot$
	such that for every name $p$ of a point $x\in X$ and every name $q$ of an open $U\in\OO(X)$
	the value $F(p,q)$ satisfies $x\in B_\bot(F(p,q))\In U$ if $x\in U$ and $B_\bot(F(p,q))=\varnothing$ otherwise.
	Since $X$ is complete (and hence has a total representation), we obtain by currying 
	a computable function $G:\In\IN^\IN\to\CC(\IN^\IN,Y_\bot)$.
	Hence, given a name $q$ of an open set $U\in\OO(X)$, we can compute
	$A:=\overline{G(q)(\IN^\IN)}\in\AA_+(Y_\bot)$ by Proposition~\ref{prop:computable-hyperspace}
	and we obtain
	\[\bigcup_{y\in A}B_\bot(y)=\bigcup_{y\in F(\IN^\IN\times\{q\})}B_\bot(y)=U.\]
	Hence, $B_\bot:Y_\bot\to\OO(X)$ is a computable Lacombe base of $X$.
\end{proof}

We could weaken the condition on completeness of $X$ even to represented spaces $X$
that admit a representation with an overt domain. In Figure~\ref{fig:bases} some
of the implications between types of bases are displayed.

Next we want to discuss how overt spaces can be characterized in terms of bases.
We recall that a represented space $X$ is called {\em overt} if $X$ is a 
computable point in $\AA_+(X)$, or in other words, if $\OO(X)^\times:=\{U\in\OO(X):U\not=\varnothing\}$
is a c.e.\ open subspace of $\OO(X)$. By $\id^\times:\OO(X)^\times\to\OO(X)$ we denote
the identity restricted to $\OO(X)^\times$. 
We also need to following natural generalization of the finite intersection property.

\begin{definition}[Compact intersection property]
	Let $X$ be a represented space. 
	\begin{enumerate}
		\item We say that $X$ has the {\em compact intersection property}, if $\bigcap K\not=\varnothing$ for all $K\in\KK_-(\OO(X)^\times)$ implies $\bigcap\OO(X)^\times\not=\varnothing$.
		\item We say that $X$ has the {\em computable compact intersection property},
		if $\bigcap K\not=\varnothing$ for all computable $K\in\KK_-(\OO(X)^\times)$ implies
		that there is a computable $x\in\bigcap\OO(X)^\times$.
	\end{enumerate}
\end{definition}

We note that the compact intersection property is satisfied by many spaces,
among them all sober spaces. This is because the premise of the compact intersection
property implies that all pairs of non-empty open subsets of $X$
intersect, which means that $X$ itself is irreducible. If it is sober, then
it follows that $X=\overline{\{x\}}$, which shows the conclusion is satisfied.
On the other hand, if $X$ is not-sober, then $X\sqcup\{0\}$ is also not sober
and yet satisfies the compact intersection property, because the premise is not satisfied.
(Here $\sqcup$ denotes the {\em coproduct}, see a formal definition before Proposition~\ref{prop:constructions-overt-prebases}.)

Now we can prove a characterization of overt Kolmogorov spaces $X$ via certain 
bases that avoid the empty set. We say that
a map $B:Y\to\OO(X)$ {\em avoids the empty set}, if $\range(B)\In\OO(X)^\times$.

\begin{theorem}[Overt spaces and bases]
	\label{thm:Kolmogorov-bases-overt}
	Let $X$ be a computable Kolmogorov space. Then the following are pairwise equivalent:
	\begin{enumerate}
		\item $X$ is overt.
		\item $X$ has a computable Lacombe base that avoids the empty set.
		\item $\id^\times:\OO(X)^\times\to\OO(X)$ is a computable Lacombe base of $X$.
	\end{enumerate}
     If $X$, additionally, satisfies the computable compact intersection property, then the following
	condition is equivalent to each of the aforementioned ones:
	\begin{enumerate}[resume]
		\item $\id^\times:\OO(X)^\times\to\OO(X)$ is a computable base of $X$.
	\end{enumerate}
\end{theorem}
\begin{proof}
	(3)$\TO$(2) is clear.\\
	(2)$\TO$(1) 
	Let $B:Y\to\OO(X)$ be a computable Lacombe base that avoids the empty set.
	We need to prove that $X$ is overt.
	Given an open set $U\in\OO(X)$, we can hence compute a set $A\in\AA_+(Y)$
	such that $U=\bigcup_{y\in A}B_y$. Now
	\[U\not=\varnothing\iff A\not=\varnothing\iff A\cap Y\not=\varnothing\]
	and the latter property is c.e.\ in $A$ since $Y$ is c.e.\ open in itself. Hence $X$ is overt.\\
	(1)$\TO$(3). Let $X$ be an overt computable Kolmogorov space.
    Given an open set $U\in\OO(X)$, we can compute the set
    \[A:=\left\{\begin{array}{ll}
    \overline{\{U\}} & \mbox{if $U\not=\varnothing$}\\
    \varnothing & \mbox{otherwise}
    \end{array}\right.\in\AA_+(\OO(X)^\times)\]
    as overtness of $X$ implies that given $\UU\in\OO(\OO(X)^\times)$ the following
    condition is c.e.:
    \[A\cap\UU\not=\varnothing\iff U\not=\varnothing\text{ and }U\in\UU.\] 
    We obtain $\bigcup_{V\in A}\id^\times(V)=U$ and hence $\id^\times$ is a computable Lacombe base.\\
    (3)$\TO$(4) Follows from Proposition~\ref{prop:general-computable-base}.\\
    (4)$\TO$(1) 
    Let us now assume that $X$ has the computable compact intersection property.
    Then there is a computable $x\in\bigcap\OO(X)^\times$ or some computable $K\in\KK_-(\OO(X)^\times)$ such that $\bigcap K=\varnothing$.
    In the first case, $U\in\OO(X)$ satisfies $U\not=\varnothing\iff x\in U$ and hence $X$ is overt. In the second case, we claim that given $U\in\OO(X)$ we can
    compute the set $K\vee U=\sat\{V\cup U:V\in K\}\in\KK_-(\OO(X)^\times)$.
    If $\id^\times$ is a computable base of $X$, then we can compute
    $A\in\AA_+(\OO(X)^\times)$ such that
    $\bigcup A=\bigcap(K\vee U)=U$,
    which implies $U\not=\varnothing\iff A\not=\varnothing$, which is a c.e.\ open
    property, as explained above. Hence, $X$ is overt in this case too.
    We still need to prove the claim about $K\vee U$.
    Firstly, we note that the map $f_U:\OO(X)^\times\to\OO(X)^\times,V\mapsto V\cup U$
    is well-defined and for $\UU\in\OO(\OO(X)^\times)$ we obtain
    \[\forall_{K\vee U}(\UU)\iff(\forall W\in K\vee U)\,W\in\UU\iff
    (\forall V\in K)\,V\cup U\in\UU\iff\forall_K(f_U^{-1}(\UU)).\]
    Finally, $\OO(X)\to\KK_-(\OO(X)^\times),U\mapsto K\vee U$ is computable
    by Proposition~\ref{prop:computable-hyperspace},
    as the function $U\mapsto\CC(\OO(X)^\times,\OO(X)^\times),U\mapsto f_U$ is computable.        
\end{proof}

It is obvious that $\id^\times:\OO(X)^\times\to\OO(X)$ is a computable
Nogina base for every represented space. Hence, the existence of such
a base cannot be used to characterize overt spaces. 
It is also easy to see that every space $X$ with two non-empty disjoint c.e.\ open
subsets $U_1,U_2\in X$ satisfies the computable compact intersection property
(since $K=\sat\{U_1,U_2\}\in\KK_-(\OO(X)^\times)$ is computable then).
This observation yields the following negative result.

\begin{corollary}
	If $X$ is a non-overt computable Kolmogorov space with two disjoint 
	non-empty c.e.\ open subsets,
	then $\id^\times:\OO(X)^\times\to\OO(X)$ is a computable Nogina base,
	which is neither a computable Lacombe base nor a computable base.
\end{corollary}

Every non-c.e.\ subspace $X\In\IN$ is an example of a non-overt computable Kolmogorov space
with two non-empty c.e.\ open subsets.
Also computable presubbases that avoid the empty set cannot be used to characterize overtness.
This holds even in the second-countable case with $Y=\IN$.

\begin{example}
Let $a,b\in A\In\IN$ with $a\not=b$ and $B_{2n}:=\{a,n\}$, $B_{2n+1}:=\{b,n\}$ for all $n\in\IN$.
Then $B:\IN\to\OO(\IN),n\mapsto B_n$ is a computable presubbase of $\IN$ and hence
we obtain a computable presubbase $B':\IN\to\OO(A),n\mapsto B_n\cap A$ of $A$ (see also Proposition~\ref{prop:constructions-overt-prebases}).
Clearly, $B'$ avoids the empty set, but $A$ is overt if and only if it is c.e.
\end{example}

In the second-countable case, computable bases $B:\IN\to\OO(X)$ that avoid
the empty set can be used to characterize overt spaces~\cite[Proposition~6.5]{BR26}.
We prove that this does not hold for arbitrary bases.
In fact, we will show that the compact intersection property in Theorem~\ref{thm:Kolmogorov-bases-overt}
was in some sense not just sufficient but also necessary.
We proceed in several steps.

\begin{lemma}
\label{lem:emptiness-presubbase}
Let $X$ be a computable Kolmogorov space that satisfies $\bigcap\OO(X)^\times=\varnothing$.
Then $\id^\times:\OO(X)^\times\to\OO(X)$ is a computable presubbase of $X$.
\end{lemma}
\begin{proof}
Since $X$ is a computable Kolmogorov space, it follows that the neighborhood map $\UU:X\to\OO\OO(X)$ is a computable embedding. We need to prove that
\[\UU^\times:X\to\OO(\OO(X)^\times),x\mapsto\UU^\times_x:=\{U\in\OO(X)^\times:x\in U\}\]
is a computable embedding too. We first note that $\OO(X)^\times$ is an open
and hence sequential subspace of $\OO(X)$.
Thus, $\OO(\OO(X)^\times)=\{\VV\cap\OO(X)^\times:\VV\in\OO\OO(X)\}$.
Since the identity ${\iota:\OO(X)^\times\to\OO(X)}$ is a computable embedding,
we obtain that $\UU^\times=\iota^{-1}\circ\UU$ is computable too. 
We need to show
that the partial inverse of $\UU^\times$ is computable.
We consider the representation $\delta_{\OO(\OO(X)^\times)}$ according
to Proposition~\ref{prop:synthetic-final}. Let $q\in\IN^\IN$
be such that $\delta_{\OO(\OO(X)^\times)}(q)=\UU^\times_x$ for some $x\in X$
and let $j\in q$. Then for every $U\in\OO(X)^\times$ \[U\in\delta_{\OO(X)^\times}(\w_j\IN^\IN)=\delta_{\OO(X)}(\w_j\IN^\IN)\cap\OO(X)^\times\TO x\in U.\]
Since $\bigcap\OO(X)^\times=\varnothing$, there is some $U\in\OO(X)^\times$
with $x\not\in U$ and hence $U\not\in\delta_{\OO(X)^\times}(\w_j\IN^\IN)$.
But this implies $\varnothing\not\in\delta_{\OO(X)}(\w_j\IN^\IN)$, since
otherwise $\delta_{\OO(X)}(\w_j\IN^\IN)=\OO(X)$ would follow.
But if $\varnothing\not\in\delta_{\OO(X)}(\w_j\IN^\IN)$ for all $j\in q$,
then $q\in\dom(\delta_{\OO\OO(X)})$ and $\delta_{\OO\OO(X)}(q)=\UU_x$.
But this means that given a name $q$ for $\UU^\times_x$, we actually
already have a name for $\UU_x$ and hence we can compute $x\in X$
and $\UU^\times$ is a computable embedding. 
\end{proof}

From this result it follows that we even obtain a computable base
if the space does not satisfy the compact intersection property.

\begin{proposition}
	\label{prop:compact-intersection}
	If $X$ is a computable Kolmogorov space that does not satisfy the
	compact intersection property, then $\id^\times:\OO(X)^\times\to\OO(X)$
	is a computable base of $X$.
\end{proposition}
\begin{proof}
	That $X$ does not satisfy the compact intersection property implies that
	$\bigcap K\not=\varnothing$ for all $K\in\KK_-(\OO(X)^\times)$ and
	yet $\bigcap\OO(X)^\times=\varnothing$.
	The latter condition implies that $\id^\times$ is a computable presubbase
	by Lemma~\ref{lem:emptiness-presubbase}.
	We prove that $\id^\times$ is even a computable base.
	We claim that the map
	$\bigcap:\KK_-(\OO(X)^\times)\to\OO(X)^\times,K\mapsto\bigcap K$	
	is well-defined and computable.
	Then it follows by Lemma~\ref{lem:closure-compact-intersection}
	that $\id^\times$ is a computable base (since it is obviously also a base).
	We now prove the claim. 
	Firstly, by the non-emptiness assumption we have $\bigcap K\not=\varnothing$
	for all $K\in\KK_-(\OO(X)^\times)$. Hence, for the well-definedness
	we still need to show that $\bigcap K$ is open, which we prove together with 
	the computability of the map.
	We obtain for $K\in\KK_-(\OO(X)^\times)$
	\[x\in\bigcap K\iff(\forall U\in K)\,x\in U\iff(\forall U\in K)\,U\in \UU^\times_x\iff\forall_K(\UU^\times_x)=1,\]
	i.e., $\chi_{\bigcap K}=\forall_K\circ \UU^\times$. 
	This proves that $\bigcap:\KK_-(\OO(X)^\times)\to\OO(X)$ is computable and 
	well-defined.
	As the identity $\iota:\OO(X)^\times\to\OO(X)$ is a computable embedding,
	we can also compute $\bigcap$ with $\OO(X)^\times$ as output space, as claimed.
\end{proof}

Finally, together with Proposition~\ref{thm:Kolmogorov-bases-overt} we obtain
the desired sufficient criterion for $\id^\times$ being a computable base,
but not a computable Lacombe base.

\begin{corollary}
	\label{cor:computable-non-Lacombe}
	If $X$ is a computable Kolmogorov space that is neither overt, nor satisfies
	the compact intersection property, then $\id^\times:\OO(X)^\times\to\OO(X)$
	is a computable base, but not a computable Lacombe base.
\end{corollary}

We still need to prove that there are actually spaces that satisfy the
requirements given in this corollary.
The natural numbers with the cofinite topology are a natural example of a 
space that does not satisfy the compact intersection property. However, this
space is overt. But a non-overt version of this space yields
a counterexample. And the space $X$ in this counterexample
is even computably second-countable.

\begin{example}[Non-overt cofinite space of natural numbers]
\label{ex:cofinite}
We consider the space $X=\IN$ with the cofinite topology,
induced by the base $B:\IN\to \OO(X)$ with
\begin{enumerate}
	\item $B_{2n}:=\IN\setminus\range(\w_n)$,
	\item $B_{2n+1}:=\varnothing$ if $n\not\in A$,
	\item $B_{2n+1}:=\IN$ if $n\in A$
\end{enumerate}
for all $n\in\IN$, where $A\In\IN$ is some set that is not c.e.\ and
$\w:\IN\to\IN^*$ is a standard numbering of words. 
We endow $X$ with the subbase representation $\delta^B$ induced by this subbase.
Then $X$ is a computably second-countable computable Kolmogorov space
that is not overt and does not satisfy the compact intersection property.
Hence, $\id^\times:\OO(X)^\times\to\OO(X)$ is a computable base of $X$
that is not a computable Lacombe base.
\end{example}
\begin{proof}
It is clear that $X$ is a computably second-countable computable Kolmogorov space.
Then $B:\IN\to\OO(X)$ is computable for this space and $B_{2n+1}\not=\varnothing\iff n\in A$. Hence $X$ is not overt. 
We have $\OO(X)^\times=\{U\In\IN:\IN\setminus U\text{ finite}\}$.
Hence, $X$ does not satisfy the finite intersection property in the sense
that $U_1\cap U_2\cap...\cap U_n\not=\varnothing$ for any choice of finitely many $U_1,...,U_n\in\OO(X)^\times$, but $\bigcap\OO(X)^\times=\varnothing$.
Hence, by Corollary~\ref{cor:computable-non-Lacombe}
it suffices to show that $\KK_-(\OO(X)^\times)$ contains only finite sets.
Since $\OO(X)^\times$ is an open subset of $\OO(X)$, it is a sequential subspace
and hence $\OO(\OO(X)^\times)=\{\UU\cap\OO(X)^\times:\UU\in\OO\OO(X)\}$.
Since every subset of $X$ is compact, in particular, the sets $B_{2n}$ are compact
and hence 
\[\UU_n:=\{U\in\OO(X)^\times:B_{2n}\In U\}=\FF_{B_{2n}}\cap\OO(X)^\times\in\OO(\OO(X)^\times)\]
for all $n\in\IN$. The sets $\UU_n$ form an open cover of $\OO(X)^\times$
and hence every compact set $K\in\KK_-(\OO(X)^\times)$ has a finite
subcover of the form $K\In\UU_1\cup...\cup\UU_m$ for some $m\in\IN$.
This implies that $K$ is finite since all the $\UU_n$ are finite. This 
completes the proof.
\end{proof}

\section{Closure properties}

In this section we provide a number of examples of computable presubbases and prebases and we 
demonstrate that these concepts can be used to derive interesting results.
All these results are essentially well-known (mostly due to the work of Schr\"oder~\cite{Sch02c}).
The purpose here is not to claim originality, but to show how these results can be easily
derived using the concepts of bases.
We start with providing a number of examples that show that computable presubbases of hyper and function spaces
occur very naturally.

\begin{proposition}[Computable presubbases]
\label{prop:presubbase-hyperspace-function-space}
Let $X$ be a represented space and $Y$ a computable Kolmogorov space. 
The following are computable presubbases:
\begin{enumerate}
\item $\UU:X\to\OO\OO(X),x\mapsto\{U\in\OO(X):x\in U\}$.
\item $\FF:\KK_-(X)\to\OO\OO(X),K\mapsto\{U\in\OO(X):K\In U\}$.
\item $\Box:\OO(X)\to\OO\KK_-(X),U\mapsto\{K\in\KK_-(X):K\In U\}$.
\item $\Diamond:\OO(X)\to\OO\AA_+(X),U\mapsto\{A\in\AA_+(X):A\cap U\not=\varnothing\}$.
\item $\triangleright:\KK_-(X)\times\OO(Y)\to\OO\CC(X,Y),(K,U)\mapsto\{f\in\CC(X,Y):f(K)\In U\}$.
\end{enumerate}
The maps $\FF$ and $\Box$ are even computable prebases. 
\end{proposition}
\begin{proof}
We have $\UU^\T=\id_{\OO(X)}$, $\Box^\T=\FF$, $\Diamond^\T=\TT$ and $\triangleright^\T=\CC\OO$
where the maps $\FF$, $\TT$, $\CC\OO$ and $\Box$ are computable embeddings according to Proposition~\ref{prop:computable-hyperspace}.
Hence, $\UU$, $\FF$, $\Box$, $\Diamond$ and $\triangleright$ are all computable presubbases.

We still need to prove that $\FF$ and $\Box$ are even computable prebases.
Given $\KK\in\KK_-\KK_-(X)$ we can compute $\sat\bigcup_{K\in\KK}K\in\KK_-(X)$ according to 
Proposition~\ref{prop:computable-hyperspace} and because
$\bigcap_{K\in\KK}\FF_K=\FF_{\sat\bigcup_{K\in\KK}K}$ it follows by Lemma~\ref{lem:closure-compact-intersection} and Proposition~\ref{prop:computable-hyperspace}
that $\FF$ is a computable prebase.
Likewise, given $\KK\in\KK_-\OO(X)$ we can compute $\bigcap_{U\in\KK}U\in\OO(X)$ and because
$\bigcap_{U\in\KK}\Box  U=\Box(\bigcap_{U\in\KK}U)$ it follows by Lemma~\ref{lem:closure-compact-intersection}
that $\Box$ is a computable prebase.
\end{proof}

In particular, we can conclude that all the involved spaces $\OO(X)$, $\KK_-(X)$, $\AA_+(X)$ and $\CC(X,Y)$ are computable Kolmogorov spaces under the given conditions
(see Corollary~\ref{cor:Kolmogorov-closure-properties}).
The maps specified in Proposition~\ref{prop:presubbase-hyperspace-function-space}
are all known as subbases of some well-known topologies on the respective
spaces. We summarize the terminology in the table in Figure~\ref{fig:subbases}.

\begin{figure}[htb]
	\begin{tabular}{lll}
		{\bf space} & {\bf subbase} & {\bf name of topology}\\\hline
		$\CC(X,Y)$ & $\triangleright$ & compact-open topology\\
		$\AA_+(X)$ & $\Diamond$ & lower Fell topology\\
		$\KK_-(X)$ & $\Box$ & upper Vietoris topology\\
		$\OO(X)$ & $\FF$ & compact-open topology\\
		$\OO(X)$ & $\UU$ & point-open topology
	\end{tabular}
\caption{Function and hyperspace topologies and their subbases.}
\label{fig:subbases}
\end{figure}

Proposition~\ref{prop:presubbase-hyperspace-function-space} can also be used as an example
to illustrate the different scopes of presubbases, prebases, and bases, respectively.

\begin{example}
\label{ex:presubbase-prebase}
Let $X$ be some represented space.
\begin{enumerate}
\item The neighborhood map $\UU$
is a computable presubbase of $\OO(X)$ 
that generates the {\em point-open topology} on $\OO(X)$.
\item The canonical computable prebase associated to $\UU$ (see Proposition~\ref{prop:computable-prebase}) is $\bigcap_\KK\UU=\FF$, i.e., the filter map $\FF$ that generates 
the compact-open topology on $\OO(X)$, whose sequentialization 
is hence the Scott topology on $\OO(X)$.
\item The map $\FF$ is a computable base of $\OO(X)$ with respect to the Scott topology
if and only if the space $X$ is consonant.
\item The map $\id_{\OO\OO(X)}$ is always a computable (Lacombe) base of $\OO(X)$ with respect to the Scott topology.
\end{enumerate}
\end{example}

These examples show that in the general case the concept of a computable presubbase and of a computable
prebase might actually be more fruitful and interesting than the concept of a computable base.
We have lots of natural presubbases and prebases that reveal useful information about our spaces. 

With the following result we provide a few natural constructions of prebases, where we
exploit that the index spaces are overt. 
Overt spaces are exactly those spaces that allow computable
projections on open sets.

\begin{proposition}[Overt space]
	\label{prop:overt-spaces}
A represented spaces $X$ is overt if and only if the projection
\[\pr_Y:\OO(Y\times X)\to\OO(Y),U\mapsto\{y\in Y:(\exists x\in X)\;(y,x)\in U\}\]
is computable for every represented space $Y$
\end{proposition}
\begin{proof}
We use the fact that sections $U_y:=\{x\in X:(y,x)\in U\}$ and products are computable by~Proposition~\ref{prop:computable-hyperspace}.\\
	``$\TO$'' $\chi_{\pr_Y(U)}(y)=\exists_X(U_y)$ as $(\exists x\in X)\;(y,x)\in U\iff U_y\not=\varnothing$.\\
	``$\Longleftarrow$'' $\exists_X(U)=\exists_X((\{0\}\times U)_0)=\chi_{\pr_Y(\{0\}\times U)}(0)$ for the space $Y=\{0\}$.
\end{proof}

Now we can prove the following closure properties for bases $B:R\to\OO(X)$
with overt index spaces $R$. We note that $\OO(X)$ is always overt, hence every computable
Kolmogorov space $X$ has a computable base $\id_{\OO(X)}:\OO(X)\to\OO(X)$ 
with an overt index space. By $X\sqcup Y$ we denote the {\em coproduct}
of two represented spaces $X$ and $Y$, which is the set $X\sqcup Y=(\{0\}\times X)\cup(\{1\}\times Y)$ with its canonical representation and by $X\sqcap Y$ we denote the {\em meet}
$X\sqcap Y=X\cap Y$, which is represented with names that are pairs, where one component
represents the point as a point in $X$ and the second component the same point as a point in $Y$. We can identify the space $X\sqcap Y$ with the subspace $\{(x,y)\in X\times Y:x=y\}$ of $X\times Y$. Hence, the results on $X\sqcap Y$ could also be derived
from the results on products and subspaces.
By $X^*$ we denote the {\em space of finite words} over $X$ with its canonical representation. 
This space can be formalized
as a countable coproduct $X^*=\bigsqcup_{n\in\IN}X^n=\bigcup_{n\in\IN}\left(\{n\}\times X^n\right)$, analogously to the finite case. 

\begin{proposition}[Constructions of overt prebases]
\label{prop:constructions-overt-prebases}
Let $X$ and $Y$ be represented spaces and let $Z\In X$.
Let $S$ and $R$ be overt represented spaces.
If $B_X:R\to\OO(X)$ and $B_Y:S\to\OO(Y)$ are computable presubbases (prebases),
then so are: 
\begin{enumerate}
\item $B_{X\times Y}:R\times S\to\OO(X\times Y),(r,s)\mapsto B_X(r)\times B_Y(s)$.
\item $B_{Y^\IN}:S^*\to\OO(Y^\IN),(s_1,...,s_n)\mapsto B_Y(s_1)\times...\times B_Y(s_n)\times Y^\IN$.
\item $B_Z:R\to\OO(Z),r\mapsto B_X(r)\cap Z$.
\item $B_{X\sqcup Y}:R\sqcup S\to\OO(X\sqcup Y),t\mapsto\left\{\begin{array}{ll}
B_X(t) & \mbox{if $t\in R$}\\
B_Y(t) & \mbox{if $t\in S$}
\end{array}\right..$
\item $B_{X\sqcap Y}:R\times S\to\OO(X\sqcap Y),(r,s)\mapsto B_X(r)\cap B_Y(s)$.
\end{enumerate}
If $B_X$ and $B_Y$ are even computable bases, then so is $B_{X\sqcup Y}$. 
Overtness of $R$ is not needed for (3).
\end{proposition}
\begin{proof}
Let $B_X$ and $B_Y$ be computable presubbases.
We obtain the transposes
\begin{enumerate}
\item $B_{X\times Y}^\T:X\times Y\to\OO(R\times S),(x,y)\mapsto B_X^\T(x)\times B_Y^\T(y)$.
\item $B_{Y^\IN}^\T:Y^\IN\to\OO(S^*),(y_n)_{n\in\IN}\mapsto\bigsqcup_{n\in\IN}\left(B_Y^\T(y_0)\times...\times B_Y^\T(y_{n-1})\right)$.
\item $B_Z^\T:Z\to\OO(R),z\mapsto B_X^\T(z)$.
\item $B_{X\sqcup Y}^\T:X\sqcup Y\to\OO(R\sqcup S),z\mapsto\left\{\begin{array}{ll}
B_X^\T(z) & \mbox{if $z\in X$}\\
B_Y^\T(z) & \mbox{if $z\in Y$}
\end{array}\right..$
\item $B_{X\sqcap Y}^\T:X\sqcap Y\to\OO(R\times S), z\mapsto B_X^\T(z)\times B_Y^\T(z)$.
\end{enumerate}
It is easy to see that all these maps are computable and we are going to show that they are computable embeddings. Hence the maps listed in the proposition
are all computable presubbases. \\
(1) and (2) We prove this in detail for the countable case of $B_{Y^\IN}$ in (2)
and we leave the finite case to the reader.
Given 
$U:=\bigsqcup_{n\in\IN}\left(B_Y^\T(y_0)\times...\times B_Y^\T(y_{n-1})\right)\in\OO(S^*)$
for some $(y_n)_{n\in\IN}\in Y^\IN$
and $n\in\IN$, we can compute the sections 
\[V:=\sec_{n+1}(U)=B_Y^\T(y_0)\times...\times B_Y^\T(y_{n})\in\OO(S^{n+1})\] 
and since $S$ is overt, and hence $S^n$ too,
we can then compute the projection $W:=B_Y^\T(y_{n})\in\OO(S)$ 
by Proposition~\ref{prop:overt-spaces} (which also holds uniformly in $n$).
From $W$ we can reconstruct $y_n$, as $B_Y^\T$ is a computable embedding.\\
(3) It immediately follows that $B_Z^\T$ is a computable embedding because $B_X^\T$
is one. Overtness of $R$ is not required here.\\
(4) Given $U\in\OO(R\sqcup S)$, where $R\sqcup S=(\{0\}\times R)\cup(\{1\}\times S)$,
we can compute the sections $U_0\in\OO(R)$ and $U_1\in\OO(S)$. 
If $U$ is in the image of $B_{X\sqcup Y}^\T$ exactly one of $U_0$ and $U_1$ is non-empty and since $R$ and $S$ are overt,
we can find out which one it is. If, for instance, $U_0\not=\varnothing$, then
$U_0=B_X^\T(x)$ and since $B_X^\T$ is a computable embedding, we can reconstruct
$x\in X$ from $U_0$. Likewise, we can reconstruct $y\in Y$ in the case of $U_1\not=\varnothing$. \\
(5) Given $U=B_X^\T(z)\times B_Y^\T(z)\in\OO(R\times S)$ we can compute
the projections $B_X^\T(z)$ and $B_Y^\T(z)$ by Proposition~\ref{prop:constructions-overt-prebases}
and hence we can reconstruct $z\in X$ as well as $z\in Y$, as $B_X^\T$ and $B_Y^\T$
are computable embeddings. This yields $z\in X\sqcap Y$.

Let now $B_X$ and $B_Y$ even be computable prebases.
We claim that in this case all the maps listed in the proposition are computable prebases too.
To this end, we need to prove that compact intersections of the given map can be computably
obtained as overt unions. \\
(1) and (2) We prove the claim for $B_{Y^\IN}$. Given $K\in\KK_-(S^*)$
we can compute the projection $K_0:=\pr_0(K)\in\KK_-(\IN)$ onto the natural number component
and we can compute some upper bound $m\in\IN$ on $\max(K_0)$. 
Hence, we can compute also the projections $K_i:=\sat(\pr_i(K))\in\KK_-(S)$ on the components
$i=1,...,m$ of $K$ by Proposition~\ref{prop:computable-hyperspace} (which holds uniformly in $i$). 
For each $i=1,...,m$ we can compute some set $A_i\in\AA_+(S)$ with $\bigcap_{s\in K_i}B_Y(s)=\bigcup_{s\in A_i}B_Y(s)$.
By Proposition~\ref{prop:computable-hyperspace} (which holds analogously for finite products) we can hence compute
$A:=A_1\times...\times A_m\in\AA_+(S^*)$ and we obtain (with implicitly bound $n\in\IN$):
\begin{eqnarray*}
\bigcap_{(s_1,...,s_n)\in K}B_{Y^\IN}(s_1,...,s_n)
&=&\bigcap_{(s_1,...,s_n)\in K}(B_Y(s_1)\times...\times B_Y(s_n)\times Y^\IN)\\
&=&\left(\bigcap_{s_1\in K_1}B_Y(s_1)\right)\times...\times\left(\bigcap_{s_m\in K_m}B_Y(s_m)\right)\times Y^\IN\\
&=&\left(\bigcup_{s_1\in A_1}B_Y(s_1)\right)\times...\times\left(\bigcup_{s_m\in A_m}B_Y(s_m)\right)\times Y^\IN\\
&=&\bigcup_{(s_1,...,s_m)\in A}B_{Y^\IN}(s_1,...,s_m).
\end{eqnarray*}
We note that the second and fourth equalities hold as the conditions on the $s_i$ are
independent and for the second equality we also use  that the saturated sets $K_i=\sat(\pr_i(K))$ can be replaced by the unsaturated sets $\pr_i(K)$ in 
the intersections by Proposition~\ref{prop:computable-hyperspace}.\\
(3) This is obvious as the intersection with $Z$ can be associated 
to a compact intersection as well as to an overt union.\\
(4) It is easy to see that given $K\in\KK_-(R\sqcup S)$ we can compute
the sections $K_0\in\KK_-(R)$ and $K_1\in\KK_-(S)$. This is because 
$K_0\In U\iff K\In U\sqcup Y$ and $K_1\In V\iff K\In X\sqcup V$ for $U\in\OO(R)$
and $V\in\OO(S)$. With these sections we obtain
\begin{eqnarray*}
	\bigcap_{t\in K}B_{X\sqcup Y}(t)
	&=&
    \left(\bigcap_{r\in K_0}B_X(r)\right)\sqcup
    \left(\bigcap_{s\in K_1}B_Y(s)\right)\\
    &=&
   \left(\bigcup_{r\in A_0}B_X(r)\right)\sqcup
   \left(\bigcup_{s\in A_1}B_Y(s)\right)
   =\bigcup_{t\in A_0\sqcup A_1}B_{X\sqcup Y}(t)
\end{eqnarray*}
where $A_0\in\AA_+(R)$ and $A_1\in\AA_+(S)$ are sets that we can compute
from $K_0$ and $K_1$ by assumption. Hence we can also compute $A:=A_0\sqcup A_1\in\AA_+(R\sqcup S)$ because $A\cap U\not=\varnothing$ for $U\in\OO(R\sqcup S)$
if and only if $A_0\cap U_0\not=\varnothing$ or $A_1\cap U_1\not=\varnothing$. \\
(5) If $K\in\KK_-(R\times S)$, then we can compute the projections $K_1:=\sat\pr_1(K)\in\KK_-(R)$ and $K_2:=\sat\pr_2(K)\in\KK_-(S)$. We obtain
\begin{eqnarray*}
	\bigcap_{(r,s)\in K}B_{X\sqcap Y}(r,s)
	&=&\bigcap_{r\in K_1}B_X(r)\cap\bigcap_{s\in K_2}B_Y(s)\\
	&=& \left(\bigcup_{r\in A_1}B_X(r)\right)\cap\left(\bigcup_{s\in A_2}B_Y(s)\right)
    =\bigcup_{(r,s)\in A_1\times A_2}B_X(r)\cap B_Y(s)
\end{eqnarray*}
and by assumption we can compute corresponding sets $A_1\in\AA_+(R)$ and $A_2\in\AA_+(S)$.

Let now $B_X$ and $B_Y$ even be computable bases.
That $B_{X\sqcup Y}$ is a base follows from 
$\OO(X\sqcup Y)=\OO(X)\sqcup\OO(Y)$.
\end{proof}

Using Proposition~\ref{prop:constructions-overt-prebases} and the computable base
$\id_{\OO(X)}:\OO(X)\to\OO(X)$ we even obtain specific
computable prebases for spaces associated to computable Kolmogorov spaces. 

\begin{corollary}[Computable prebases]
\label{cor:constructions-computable-prebases}
Let $X$ and $Y$ be computable Kolmogorov spaces and let $Z\In X$. Then the following
are computable prebases:
\begin{enumerate}
\item $B_{X\times Y}:\OO(X)\times \OO(Y)\to\OO(X\times Y),(U,V)\mapsto U\times V$.
\item $B_{Y^\IN}:\OO(X)^*\to\OO(Y^\IN),(U_1,...,U_n)\mapsto U_1\times...\times U_n\times Y^\IN$.
\item $B_Z:\OO(X)\to\OO(Z),U\mapsto U\cap Z$.
\item $B_{X\sqcup Y}:\OO(X)\sqcup \OO(Y)\to\OO(X\sqcup Y),U\mapsto U$.
\item $B_{X\sqcap Y}:\OO(X)\times \OO(Y)\to\OO(X\sqcap Y),(U,V)\mapsto U\cap V$.
\end{enumerate}
$B_{X\sqcup Y}$ is even a computable base.
\end{corollary}

Using Theorem~\ref{thm:Kolmogorov-bases}, Corollary~\ref{cor:constructions-computable-prebases} and
Proposition~\ref{prop:presubbase-hyperspace-function-space}
we can directly conclude
that computable Kolmogorov spaces have very nice closure properties.
We note that for this proof we only need the statements on computable
prebases that were relatively easy to establish. 
By $\AA_-(X)$ we denote
the space of closed subsets of $X$ represented as complements of open
sets in $\OO(X)$. This space is hence computably isomorphic to $\OO(X)$. 
We also consider the spaces $\AA(X):=\AA_+(X)\sqcap\AA_-(X)$
and $\KK(X):=\AA_+(X)\sqcap\KK_-(X)$ (we note that the latter contains the compact, closed and saturated sets).

\begin{corollary}[Closure properties of computable Kolmogorov spaces]
\label{cor:Kolmogorov-closure-properties}
If  $X$ and $Y$ are computable Kolmogorov spaces, then so  are
\begin{enumerate}
\item  $X\times Y$, $X\sqcup Y$, $X\sqcap Y$,  $Y^\IN$, and every subspace of $X$.
\end{enumerate}
Let $X$ be a represented space. If $Y$ is a computable Kolmogorov space, then so are
\begin{enumerate}[resume]
\item $\CC(X,Y)$, $\OO(X)$, $\AA_+(X)$, $\AA_-(X)$, $\AA(X)$, $\KK_-(X)$ and $\KK(X)$.
\end{enumerate}
Analogous statements hold for continuous Kolmogorov spaces.
\end{corollary}

In fact, some of these results can also be derived from each other in various
ways. For instance, the statement on $\CC(X,Y)$ applied to the special case
of $X=\IN$ yields the statement for $Y^\IN$ and the special case of $Y=\IS$
yields the statements for $\OO(X)$ (and $\AA_-(X)$). 
This implies the statement for $\OO\OO(X)$, subspaces of which are computably
isomorphic to $\AA_+(X)$ and $\KK_-(X)$, respectively. 
We obtain the following conclusion (see \cite{Sch02c,Sch02} for the topological version)
using the fact that our product and function
space constructions satisfy evaluation and currying properties (see~\cite{Wei00,Wei87}).

\begin{corollary}[Schr\"oder 2002]
\label{cor:cartesian-closed-Kolmogorov}
The category of computable Kolmogorov spaces with computable maps as morphisms
is cartesian closed.
\end{corollary}

For Corollaries~\ref{cor:Kolmogorov-closure-properties} and \ref{cor:cartesian-closed-Kolmogorov}
it was sufficient to have computable presubbases for the respective spaces. 
The fact that we have even computable prebases for some of these spaces allows further
conclusions regarding the underlying topologies.
Using Corollaries~\ref{cor:computable-prebase} and \ref{cor:constructions-computable-prebases}
we obtain the following conclusions. 
By $\OO(X)\otimes\OO(Y)$ and $\bigotimes_{i\in\IN}\OO(X)$ we denote
the respective {\em product topologies}, by $\OO(X)|_Z$ the {\em subspace topology}
of $Z\In X$ and by $\OO(X)\wedge\OO(Y)$ the {\em meet topology} 
generated by the base $B_{X\sqcap Y}$.

\begin{corollary}[Topologies and set-theoretic constructions]
\label{cor:topology-set-theory}
Let $X$ and $Y$ be admissibly represented $\T_0$ spaces and let $Z\In X$ be a subspace. Then:
\begin{enumerate}
\item $\OO(X\times Y)=\seq(\OO(X)\otimes\OO(Y))$,
\item $\OO(X^\IN)=\seq(\bigotimes_{i\in\IN}\OO(X))$,
\item $\OO(Z)=\seq(\OO(X)|_Z)$,
\item $\OO(X\sqcap Y)=\seq(\OO(X)\wedge\OO(Y))$.
\end{enumerate}
\end{corollary}

This result
follows from Corollary~\ref{cor:constructions-computable-prebases} in 
combination with Corollary~\ref{cor:computable-prebase}.
With the help of Theorem~\ref{thm:represented-Kolmogorov}
we can transfer (1) and (2) of this result then to arbitrary represented $\T_0$ spaces,
since $\delta\mapsto\delta^\bullet$ commutes with products~\cite{Sch02}.

\section{Hyperspace and function space topologies}

In this section we continue to discuss what can be concluded from our results
regarding hyperspace and function space topologies. Once again, these
results are well-known by the work of Schröder~\cite{Sch02c} and the purpose
here is to demonstrate how these results can be obtained with the
help of the concept of a computable presubbase or prebase.

As a preparation of our results we need
the following lemma, which shows
that the subbases $\Diamond$ and
$\triangleright$ satisfy the convergent intersection property.
This allows us to apply Corollary~\ref{cor:convergent-intersection}.
It is easy to see that the convergent intersection property also holds for $\Box$, but we are not going to use this fact.

\begin{lemma}[Countable intersections]
\label{lem:countable-intersections}
Let $X$ and $Y$ be represented space.
Then the presubbases
\begin{enumerate}
	\item $\Diamond:\OO(Y)\to\OO\AA_+(Y)$ and 
	\item $\triangleright:\KK_-(X)\times\OO(Y)\to\OO\CC(X,Y)$
\end{enumerate}
satisfy the convergent intersection property.
\end{lemma}
\begin{proof}
Let $(U_n)_{n\in\IN_\infty}$ and $(K_n)_{n\in\IN_\infty}$
be converging sequences in $\OO(Y)$ and
$\KK_-(X)$, respectively.
We prove that there exists some $k\in\IN$ such that
\begin{enumerate}
	\item $\bigcap_{n\in\IN_\infty}\Diamond U_n=\Diamond U_\infty\cap\bigcap_{n=0}^k\Diamond U_n$.
	\item $\bigcap_{n\in\IN_\infty}(K_n\triangleright U_n)=(K_\infty\triangleright U_\infty)\cap\bigcap_{n=0}^k(K_n\triangleright U_n)$.
\end{enumerate}
By Proposition~\ref{prop:Scott-convergence} convergence of $(U_n)_{n\in\IN_\infty}$ implies 
	that there is a $k\in\IN$ with $U_\infty\In\left(\bigcap_{n\geq k} U_n\right)^\circ$. In particular, 
	$U_\infty\In U_n$ for all $n\geq k$.
The inclusion ``$\In$'' in (1) and (2) is clear 
in both cases and we need to show why the inverse inclusions hold.\\
(1) Let $A\in\Diamond U_\infty\cap\bigcap_{n=0}^k\Diamond U_n$.
Then, in particular, $A\cap U_\infty\not=\varnothing$ and hence
$A\cap U_n\not=\varnothing$ for all $n\geq k$. This proves the claim.\\
(2) 
Let $f\in(K_\infty\triangleright U_\infty)\cap\bigcap_{n=0}^k(K_n\triangleright U_n)$ be a continuous function $f:X\to Y$. Then, in particular,
$f(K_\infty)\In U_\infty$.
As $(f(K_n))_{n\in\IN}$ converges in $\KK_-(Y)$, we obtain (see the independently proved statement in Theorem~\ref{thm:hyperspace-functionspace-topologies}~(2)) that there exists
a $k'\in\IN$ with 
$f(K_n)\In U_\infty$ for $n\geq k'$ and hence
$f(K_n)\In U_n$ for all $n\geq k'':=\max(k,k')$. This proves the claim.
\end{proof}

We use the terminology for topologies introduced in the table in Figure~\ref{fig:subbases}
and we add some terminology:
\begin{enumerate}
\item  The {\em upper Fell topology}
	on the hyperspace $\AA_-(X)$ is the topology inherited (by complementation) from the compact-open topology on $\OO(X)$ and 
	the {\em Fell topology} on $\AA(X)$ is the meet topology of the lower and upper
	Fell topologies.
	\item 
	The {\em Vietoris topology} on $\KK(X)$ is the meet topology of the lower Fell topology and the upper Vietoris topology.
\end{enumerate}
From Proposition~\ref{prop:presubbase-hyperspace-function-space} we can obtain information on the topologies of the respective hyperspaces,
a fact which we already indicated in Example~\ref{ex:presubbase-prebase}.

\begin{theorem}[Hyperspace and function space topologies]
\label{thm:hyperspace-functionspace-topologies}
Let $X$ be a represented space and let $Y$ be an admissibly represented $\T_0$ space.
\begin{enumerate}
\item $\OO(X)$ is endowed with the Scott topology, which is the sequentialization of the compact-open topology.
\item $\KK_-(X)$ and $\KK(X)$ are endowed with topologies that are the sequentializations of the 
upper Vietoris and the Vietoris topology, respectively.
\item $\AA_+(X)$ and $\AA(X)$ are endowed with topologies that are the sequentializations 
of the lower Fell and the Fell topology, respectively.
\item $\CC(X,Y)$ is endowed with the sequentialization of the compact-open topology. 
\end{enumerate}
In particular, all the mentioned spaces are admissibly represented with respect to the given topologies.
\end{theorem}
\begin{proof}
(1) This is the statement of Corollary~\ref{cor:Scott}.
That $\OO\OO(X)$ is the sequentialization of the topology generated by the sets $\FF_K$ for (saturated) compact
$K\In X$ follows also from Proposition~\ref{prop:presubbase-hyperspace-function-space} in combination with Corollary~\ref{cor:computable-prebase}.\\
(2) That $\OO\KK_-(X)$ is the sequentialization of the topology generated by the sets
$\Box U$ for open $U\In X$ follows from Proposition~\ref{prop:presubbase-hyperspace-function-space} in combination with Corollary~\ref{cor:computable-prebase}.\\
(3) That $\OO\AA_+(X)$ is the sequentialization of the lower Fell topology generated by the sets
$\Diamond U$ for open $U\In X$ follows from Theorem~\ref{thm:Schroder-admissible}, 
Corollary~\ref{cor:convergent-intersection}
and Lemma~\ref{lem:countable-intersections}. 
With the help of Corollary~\ref{cor:topology-set-theory}
this implies that $\KK(X)=\AA_+(X)\sqcap\KK_-(X)$ and $\AA(X)=\AA_+(X)\sqcap\AA_-(X)$ are endowed with the sequentializations of the Vietoris and Fell
topologies, respectively.\\
(4) If $Y$ is an admissibly represented $\T_0$ space
then Theorem~\ref{thm:represented-Kolmogorov} allows us to replace it by a computable
Kolmogorov space without affecting the topological structure. In fact, the resulting representation
is in the same continuous equivalence class as the original and hence so are the corresponding
representations of $\CC(X,Y)$.
By Theorem~\ref{thm:Schroder-admissible}, 
Corollary~\ref{cor:convergent-intersection}
and Lemma~\ref{lem:countable-intersections} 
we know that $\OO\CC(X,Y)$ is the sequentialization of the compact-open topology on $\CC(X,Y)$.
\end{proof}

The examples of the spaces $\AA_+(X)$ and $\CC(X,Y)$ show 
that even in cases where we have only a computable presubbase 
(and not a computable prebase)
it might happen that we obtain as a topology on our space the sequentialization of the topology generated 
by the presubbase.

\section{Epilogue}

We close this section with a brief recap of our results from a somewhat different perspective.
The results in this article can be viewed as
exploiting a Galois connection between maps
\[\delta\mapsto B_\delta\mbox{ and }B\mapsto\delta^B.\]
Let us fix a set $X$ and let us denote by $\mathsf{REP}_0$ the set of
representations $\delta:\In\IN^\IN\to X$ with a $\T_0$ final topology and let us denote by $\mathsf{PRE}_0$ the set
of presubbases $B:\In\IN^\IN\to2^X$ that induce a $\T_0$ topology on $X$.\footnote{We call $B:\In\IN^\IN\to 2^X$ a presubbase
if $(B_y)_{y\in D}$ with $D:=\dom(B)$ is a presubbase in the sense of our definition.}
\begin{enumerate}
\item The map $\Delta:\mathsf{PRE}_0\to\mathsf{REP}_0,B\mapsto\delta^B$ assigns to any presubbase $B$ the presubbase
representation $\delta^B:\In\IN^\IN\to X$. 
\item The map $\nabla:\mathsf{REP}_0\to\mathsf{PRE}_0,\delta\mapsto B_\delta$ assigns to every representation $\delta$ of $X$ 
the induced representation $\delta_{\OO(X)}$ as presubbase $B_\delta:\In\IN^\IN\to2^X$. \end{enumerate}
The presubase $B_\delta$ corresponds to the base $\id:\OO(X)\to\OO(X)$.
The maps $\Delta$ and $\nabla$ induce an antitone Galois connection, if we define a computable reducibility $\leq$ for $\mathsf{PRE}_0$ as for $\mathsf{REP}_0$ (for $\mathsf{REP}_0$ we use
the usual computable reducibility, as introduced before Theorem~\ref{thm:represented-Kolmogorov}).
Namely for $B_1,B_2\in\mathsf{PRE}_0$ we define analogously $B_1\leq B_2$ if there is some computable $F:\In\IN^\IN\to\IN^\IN$ with $B_1(p)=B_2F(p)$ for all $p\in\dom(B_1)$.
The proof of the following result is then straightforward and a simple
consequence of properties of the transposition operation.

\begin{theorem}[Galois connection for representations and presubbases]
\label{thm:Galois-representation-presubbase}
Fix a representable space $X$. Then for every $\delta\in\mathsf{REP}_0$ and $B\in\mathsf{PRE}_0$ we have
\[\delta\leq\delta^B\iff B\leq B_\delta.\]
\end{theorem}
\begin{proof}
Let $\delta:\In\IN^\IN\to X$ be a representation with a $\T_0$ topology $\OO(X)$
and induced representation $\delta_{\OO(X)}$ of $\OO(X)$
and let $B:\In\IN^\IN\to2^X$ be a presubbase that induces a $\T_0$ topology on $X$.
Then we obtain:
\begin{eqnarray*}
\delta\leq\delta^B
&\iff& B^\T:\In X\to\OO(\IN^\IN) \mbox{ is computable with respect to $\delta$}\\
&\iff& B:\In\IN^\IN\to\OO(X)\mbox{ is computable with respect to $\delta_{\OO(X)}$}\\	
&\iff& B\leq B_\delta,
\end{eqnarray*}
which proves the claim.
\end{proof}

We also obtain that $\range(\Delta)$ is exactly the set of representations of $X$ 
that turn $X$ into a computable Kolmogorov space (up to computable equivalence)
and $\range(\nabla)$ is exactly the set of representations of 
$\qcb_0$ topologies for $X$ (up to computable equivalence).

Any antitone Galois connection yields corresponding closure operators, sometimes
called {\em monads}. In the case of our Galois connection we obtain:
\begin{enumerate}
	\item $\Delta\circ\nabla:\mathsf{REP}_0\to\mathsf{REP}_0,\delta\mapsto\delta^\bullet:=\delta^{B_\delta}$ and 
	\item $\nabla\circ\Delta:\mathsf{PRE}_0\to\mathsf{PRE}_0,B\mapsto B_\bullet:=B_{\delta^B}$,
\end{enumerate}

We summarize some properties of these closure operators (see Theorem~\ref{thm:represented-Kolmogorov}, Corollary~\ref{cor:presubbase-theorem}
and Proposition~\ref{prop:identity-base}):

\begin{enumerate}
\item $\delta\mapsto\delta^\bullet$ maps every represented $\T_0$ space $(X,\delta)$ to a computable Kolmogorov space $X^\bullet:=(X,\delta^\bullet)$ with $\delta\leq\delta^\bullet$.
The induced underlying final topologies and the representations thereof are preserved by
this operation (up to computable equivalence).
\item $B\mapsto B_\bullet$ maps any presubbase $B$ of a $\T_0$ topology on $X$ to a computable base $B_\bullet$ of the sequentialization $\seq(\tau)$ of the topology $\tau$ 
that is generated by the compact intersections $\bigcap_\KK B$. In particular, 
$\range(B)\In\range(B_\bullet)$.
The operation preserves the induced presubbase representations of $X$ (up to 
computable equivalence).
\end{enumerate}

Since both application are closure operators, a double application does not yield
anything new, i.e., $\delta^{\bullet\bullet}\equiv\delta^\bullet$ and
$\range(B_{\bullet\bullet})=\range(B_\bullet)$.
We give two examples of the action of these closure operator that indicates that
these are natural operations. The first example is taken from Example~\ref{ex:presubbase-prebase}, the second one
is due to Schröder (personal communication).

\begin{example}[Action of closure operators]\ 
\begin{enumerate}
	\item 
If we start with the presubbase $B$ for the point-open topology on $\OO(X)$ then we obtain a base $B_\bullet$ of the Scott topology. Both bases induce computably equivalent presubbase representations of $\OO(X)$.
\item
If we start with a decimal representation $\rho$ of the real numbers $\IR$, then we obtain a representation $\rho^\bullet$ that is computably equivalent to the Cauchy representation of the reals. Both representations have the same final topology (the Euclidean topology).
\end{enumerate}
\end{example}

In light of the Galois connection between presubbases and representations we 
can interpret our results as follows. A lot of results that were originally proved 
with the help of representations and, for instance, the closure operator $\delta\mapsto\delta^\bullet$
by Schröder (for example the topological version of Corollary~\ref{cor:cartesian-closed-Kolmogorov} in~\cite[Section~4.3]{Sch02c})
can equally well be derived with the help of presubbases.
Both perspectives correspond to different sides of the same medal and they highlight
different aspects of the same mathematical structure.

The purpose of this article was to demonstrate the usefulness of considering presubbases.
The benefit of considering presubbases is that they allow us to talk
about topologies using their usual subbases (see, e.g., Proposition~\ref{prop:presubbase-hyperspace-function-space})
and the usual constructions of new subbases from given ones
(see, e.g., Proposition~\ref{prop:constructions-overt-prebases}).
This enriches the mathematical vocabulary that we have to 
discuss computable topology.

\bibliographystyle{alphaurl}
\bibliography{C:/Users/Documents/Spaces/Research/Bibliography/lit}

\section*{Acknowledgments}
This work was supported by the German Research Foundation (DFG, Deutsche Forschungsgemeinschaft) -- project number 554999067, by the National Research Foundation of South Africa (NRF) -- grant number 151597, by
the Alexander-von-Humboldt Foundation and 
by the Research Institute for Mathematical Sciences,
an international joint research center based at Kyoto University.
We would like to thank Matthias Schröder for discussions on the content of this
article that have helped to streamline some proofs and their presentations.
Likewise, we would like to thank the reviewer for their careful proofreading that also
helped us to improve the article.
\end{document}